\documentclass[a4paper,11pt]{amsart}

\usepackage[latin1]{inputenc}   
\usepackage{amssymb,bbm,amsmath,amsthm,mathrsfs}
\usepackage{graphicx}
\usepackage[all]{xy}
\usepackage[colorlinks=true,pagebackref]{hyperref}

\newcommand{\C}{\ensuremath{\mathbb{C}}}

\newcommand{\ra}{\ensuremath{\rightarrow}}
\newcommand{\lra}{\ensuremath{\longrightarrow}}
\newcommand{\D}{\ensuremath{\partial}}
\newcommand{\calo}{\ensuremath{\mathcal{O}}} 

\DeclareMathOperator{\Sing}{Sing}

\DeclareMathOperator{\Der}{Der}
\DeclareMathOperator{\depth}{depth}

\newcommand{\mc}[1]{\ensuremath{\mathcal{#1}}}   
\newcommand{\ms}[1]{\ensuremath{\mathscr{#1}}}   
\newcommand{\mf}[1]{\ensuremath{\mathfrak{#1}}}   

 \newcommand{\komm}[1]{}   
\long\def\ignore#1\recognize{}    

\theoremstyle{theorem}
\newtheorem{theorem}{Theorem}        
\newtheorem{lemma}{Lemma}
\newtheorem{corollary}{Corollary}
   
\newtheorem{proposition}{Proposition}   

\newtheorem{Qu}{Question}
\newtheorem{Conj}{Conjecture}

\newtheorem*{thma}{Theorem A}
\newtheorem*{thmb}{Proposition B}
\newtheorem*{thmc}{Proposition C}

\theoremstyle{remark}
\newtheorem{remark}{Remark}
\newtheorem{example}{Example}

\theoremstyle{definition}
\newtheorem{definition}{Definition} 

\title{Characterizing normal crossing hypersurfaces}
\author{Eleonore Faber}
\address{
Department of Computer and Mathematical Sciences,
University of Toronto at Scarborough,
Toronto, ON M1A 1C4,
Canada
}
\email{efaber@math.toronto.edu}
\thanks{
\noindent The author has been supported by a For Women in Science award 2011 of L'Or{\'e}al Austria, the Austrian commission for UNESCO and the Austrian Academy of Sciences and by
the Austrian Science Fund (FWF)
in frame of projects J3326 and P21461. \\
2010 Mathematics Subject Classification: 32S25, 32S10, 32A27, 14B04 \\
{\it Keywords}: normal crossing divisor, free divisor, logarithmic derivations, logarithmic differential forms, Jacobian ideal, residue} 

\begin{document}

\begin{abstract}
The objective of this article is to give an effective algebraic characterization of normal crossing hypersurfaces in complex manifolds. It is shown that a divisor (=hypersurface) has normal crossings
if and only if it is a free divisor, has a radical Jacobian ideal and a smooth normalization. Using K.~Saito's theory of free divisors, also a characterization in terms of logarithmic differential
forms and vector fields is found. Finally, we give another description of a normal crossing divisor in terms of the logarithmic residue using recent results of M.~Granger and M.~Schulze.
\end{abstract}

\maketitle

 \tableofcontents

\section{Introduction}

Consider a divisor (=hypersurface) $D$ in a complex manifold $S$ of dimension $n$. Then $D$ is said to have \emph{normal crossings} at a point $p$ if locally at $p$ there exist complex coordinates $(x_1, \ldots, x_n)$ such that
$D$ is defined by
the equation $x_1 \cdots x_m=0$ for some $0 \leq m \leq n$. In general there is no algorithm to find these coordinates. Hence the question arises: is there an \emph{effective algebraic}
characterization of a divisor with normal
crossings? \\

Normal crossing divisors appear in many contexts in algebraic and analytic geometry, for example in the embedded resolution of singularities \cite{Hi}, in compactification problems \cite{FM,DP} or in cohomology
computations \cite{Deligne71}.
However, given an (algebraic or analytic) variety, it is not clear how to determine effectively if this variety has normal crossings:
only in case the decomposition into irreducible components is known, the normal crossings property can be checked rather easily (see e.g. \cite{Bodnar04}). 
\\

The main goal of this article is to derive an effective algebraic criterion for a normal crossing divisor in a complex manifold.
By ``effective'' is meant that one should be able to decide from data derived from a local defining equation of the divisor whether it has normal crossings at a point. The guiding
idea for our investigations is that the singular locus, that is given by the Tjurina algebra, carries all information about the geometric properties of a
divisor. Here we were inspired by work of Mather--Yau about isolated hypersurface singularities \cite{MY} and Gaffney--Hauser in a more general setting \cite{GH}. 
On the other hand the tangent behaviour along the divisor, via so-called logarithmic vector fields, will give us means to control the normal crossings property. Here the rich theory of logarithmic
vector fields (differential forms), initiated by K.~Saito in the 1980's \cite{Saito80}, will be the other main ingredient for an algebraic criterion characterizing normal crossing divisors. Saito
introduced the notion of free divisor (a certain generalization of normal crossing divisor), which appears in different areas:  for example in deformation theory as discriminants \cite{Saito81,Looijenga84,Aleksandrov86,MondvanStraten01,Buchweitz06}, in combinatorics as free hyperplane arrangements \cite{Terao80,OrlikTerao92}, related to prehomogeneous vector spaces \cite{BuchweitzMond,GrangerMondSchulze11}  or in connection with the logarithmic comparison problem 
\cite{CNM96,CN09}. \\
Since a normal crossing divisor is in particular free, one is led to impose additional conditions on free divisors in order to single out the ones with normal crossings. It turns out that the radicality of the Jacobian ideal is the right property.\\

The main result is:
\begin{thma}[Thm.~\ref{Thm:radikalJacobi}]
A divisor in a complex manifold has normal crossings at a point if and only if it is free with radical Jacobian ideal at that point and its normalization is smooth.
\end{thma}
Since there is an interpretation of free divisors by
their Jacobian ideals (due to Aleksandrov \cite{Aleksandrov90}, also see \cite{Terao83,Simis06}), one thus obtains a purely algebraic characterization of normal crossing divisors.  The proof uses a
Theorem of R.~Piene about ideals in desingularizations \cite{Piene79} and also results of Granger and Schulze about the dual logarithmic residue, see \cite{GrangerSchulze11}.  The condition on the normalization is
technical and we do not know if it is necessary in general: we show that in some special cases (Gorenstein singular locus, normal irreducible components) no additional properties of the normalization
have to be required. \\

Moreover,  two other characterizations of normal crossing
divisors in terms of logarithmic differential forms (resp. vector fields) and the logarithmic residue are shown.
\begin{thmb}[Prop.~\ref{Thm:ncequivclosedforms}]
A divisor $D$ in a complex manifold $S$ has normal crossings at a point $p$ if and only if $\Omega^1_{S,p}(\log D)$, its module of logarithmic 1-forms, is free and has a basis of closed forms. This
is also equivalent to the condition that $\Der_{S,p}(\log D)$, its module of logarithmic derivations, is free and has a basis of commuting vector fields.
\end{thmb}

This result is based on Saito's theory of logarithmic differential forms. Here already the so-called logarithmic residue is used, which was also introduced by Saito and further studied by Aleksandrov
\cite{Aleksandrov05} and Aleksandrov--Tsikh \cite{AleksandrovTsikh01}. The above proposition follows from a slight modification of a Theorem of Saito (see Thm.~\ref{Thm:closednc}). 

\begin{thmc}[Prop.~\ref{Thm:irreduzibel}]
A divisor in a complex manifold has normal crossings if and only if it is free, has smooth normalization and the residues of its logarithmic 1-forms are holomorphic on the normalization.
\end{thmc}
The final result
makes use of the dual logarithmic residue, introduced by Granger and Schulze \cite{GrangerSchulze11}. \\

The article contains the following: in section \ref{Sec:singularities} singularities of normal crossing divisors are studied in order to prove Theorem A. Free divisors are introduced via the Jacobian
ideal characterization due to Aleksandrov. First it is shown that a curve in a complex manifold $S$ of dimension 2 has normal crossings at a point $p$  if and only if its Jacobian ideal is radical of
depth 2 in $\calo_{S,p}$ (Prop.~\ref{Prop:radikaldim2}). This is generalized in Prop.~\ref{Prop:Gorensteinsingular} to the case where $D$ is a divisor in a complex manifold $S$ of dimension $n$ having
reduced Gorenstein singular locus. Here we show that $\calo_{\Sing D,p}$ is a reduced  Gorenstein ring of dimension $n-2$  if and only if $D$ has normal crossings and $(\Sing D,p)$ is smooth.  Then
the
general case of the theorem is proven (using results by R.~Piene, M.~Granger and M.~Schulze and from \cite{Faber12}). In section \ref{Sub:radicalJacobiallg} radical Jacobian ideals are investigated,
in particular we pose the question, which radical ideals can be Jacobian ideals. Our main result in this direction is that if a divisor $D$ has a reduced singular locus of codimension $k$ which is
also a complete intersection, then $D$ is analytically trivial along its singular locus and isomorphic to a $k$-dimensional $A_1$-singularity (Prop.~\ref{Prop:jacobivollstdsallg}). Here, after
computation of examples, further questions and conjectures are raised. \\
In the next section, Saito's theory of logarithmic differential forms and vector fields is briefly recalled, in particular the notion of logarithmic residue. This is used to prove Proposition B. \\
In the last section we recall the dual logarithmic residue and the results of Granger and Schulze which lead to Proposition C. It is also shown that a divisor with normal irreducible components is
free and has weakly holomorphic logarithmic residue if and only if it has normal crossings, without any condition on the normalization (Lemma  \ref{Lem:normalsmooth}). Furthermore, some results on
divisors with weakly holomorphic residue are collected. We close this section with a few comments on divisors with normal crossings in codimension 1. \\
The results in this article form part of the author's Ph.D. thesis at the University of Vienna.

\section{Singularities of normal crossing divisors}  \label{Sec:singularities}

We work in the complex analytic category. The main objects of our study are divisors (=hypersurfaces) in  complex manifolds. We write $(S,D)$ for a fixed divisor $D$ in an $n$-dimensional complex
manifold $S$.
Denote by $p$ a point in $S$ and by $x=(x_1, \ldots, x_n)$ the complex coordinates of $S$ around $p$. The divisor $(D,p)$ will then be defined locally by an equation $\{
h_p(x_1, \ldots, x_n)=0 \}$ where $h_p \in \mc{O}_{S,p}$ (if the context is clear we omit the subscript $_p$). Note that we will always assume that $h$ is reduced! 
The divisor $D$ has \emph{normal crossings} at a point $p$ if one can find complex coordinates $(x_1, \ldots, x_n)$ at $p$ such that the defining equation $h$ of $D$ is $h=x_1 \cdots x_k$ for $k
\leq n$. We also say that $(D,p)$ is a \emph{normal crossing divisor}. 
The \emph{Jacobian ideal} of $h$ is denoted by $J_{h}=(\D_{x_1} h, \ldots, \D_{x_n}h) \subseteq \mc{O}_{S,p}$. 
The image of $J_h$ under the  canonical epimorphism that sends $\mc{O}_{S,p}$ to $\mc{O}_{D,p}=\mc{O}_{S,p}/(h)$ is denoted by $\widetilde{J_h}$. 
The associated analytic coherent ideal sheaves are
denoted by $\mc{J} \subseteq \mc{O}_{S}$ and $\widetilde{\mc{J}}$ in $\mc{O}_{D}$.   
The \emph{singular locus} of $D$ is denoted by $\Sing D$ and is defined by the ideal
sheaf $\widetilde{\mc{J}} \subseteq  \mc{O}_{D}$. The local ring of $\Sing D$ at a point $p$ is denoted by 
$$\mc{O}_{\Sing D,p}=\mc{O}_{S,p}/((h)+ J_h)=\mc{O}_{D,p}/\widetilde{J_h} . $$
Sometimes $\mc{O}_{\Sing D,p}=\C\{x_1, \ldots, x_n\}/(h,\D_{x_1}h, \ldots, \D_{x_n}h)$ is also called the
\emph{Tjurina algebra}. Note that we \emph{always} consider $\Sing D$ with the (possibly non-reduced) structure given by the Jacobian ideal of $D$. Hence in general $(\Sing
D,p)$ is a complex space germ and not necessarily reduced. We
often say that $\Sing D$ is \emph{Cohen--Macaulay}, which means that $\mc{O}_{\Sing D,p}$ is Cohen--Macaulay for all points $p \in \Sing D$. 
For facts about local analytic geometry we refer to \cite{dJP,Narasimhan}, about commutative algebra to \cite{Matsumura86}.

\begin{definition}
Let $D$ be a divisor in a complex manifold $S$ of dimension $n$ that is locally at a point $p$ given by $h=0$. Then $D$ is called \emph{free} at $p$ if either $D$ is smooth at $p$ or $\calo_{\Sing D,p}$ is a Cohen--Macaulay ring of dimension $n-2$.  The divisor $D$ is a \emph{free divisor} if it is free at each $p \in S$.
\end{definition}

\begin{remark}
Usually, free divisors are defined via logarithmic derivations, see section \ref{Sec:logdiff}. The equivalence of the two characterizations was proven by A.~G.~Aleksandrov in
\cite{Aleksandrov90}. 
\end{remark}

Hence free divisors are either smooth or non-normal. It is easy to see that normal crossing divisors are free (see proof of Theorem \ref{Thm:radikalJacobi}). Thus one has to impose an additional condition on the Jacobian ideal to ensure that a given free divisor really has normal crossings. In order to get an idea of the right property, look at some examples.

\begin{example} (1) Let $D$ be the cone in $\C^3$, given
by the equation $z^2=xy$. It does not have normal crossings at the origin
but the Jacobian ideal $J_{h,0}=(z,x,y)$ is clearly radical and
$\mc{O}_{\C^3,0}/(x,y,z) \cong \C$ is Cohen--Macaulay. However, the depth of
$\mc{O}_{\C^3,0}/J_{h,0}$ is $0$ and thus too small. \\
(2)  Let $S=\C^3$ and $D$ be the ``4-lines'' defined by $h=xy(x+y)(x+yz)$, see fig.~\ref{Fig:freediv} and cf. \cite{CN02,Narvaez08}. This divisor $D$ is
free.
Its Jacobian ideal is the intersection of the three primary ideals $(x+y,z-1)$,  $(x,z)$ and 
$(y^4,2xy^2z+y^3z+3x^2y+2xy^2,4x^2yz-3y^3z+2x^3-5x^2y-6xy^2)$ and  is not radical (the radical $\sqrt{J_h}$ is $(x+y,z-1) \cap (x,z) \cap (x,y)$). Also
$D$ does not have normal crossings at the origin. \\
(3) The Hessian deformation of an $E_8$ curve (see \cite{Damon02}, for a different interpretation see \cite{Sekiguchi08}) is a divisor in $\C^3$ defined by $h=y^5+z^3+xy^3z$. It does not have normal
crossings at $0$: it is irreducible and free at $0$ but $J_h$ is not radical. The reduced Jacobian ideal is $(y,z)$, the $x$-axis, cf.~fig.~\ref{Fig:freediv}.
\begin{figure}[!h]
\begin{tabular}{c@{\hspace{1.5cm}}c}
\includegraphics[width=0.39 \textwidth]{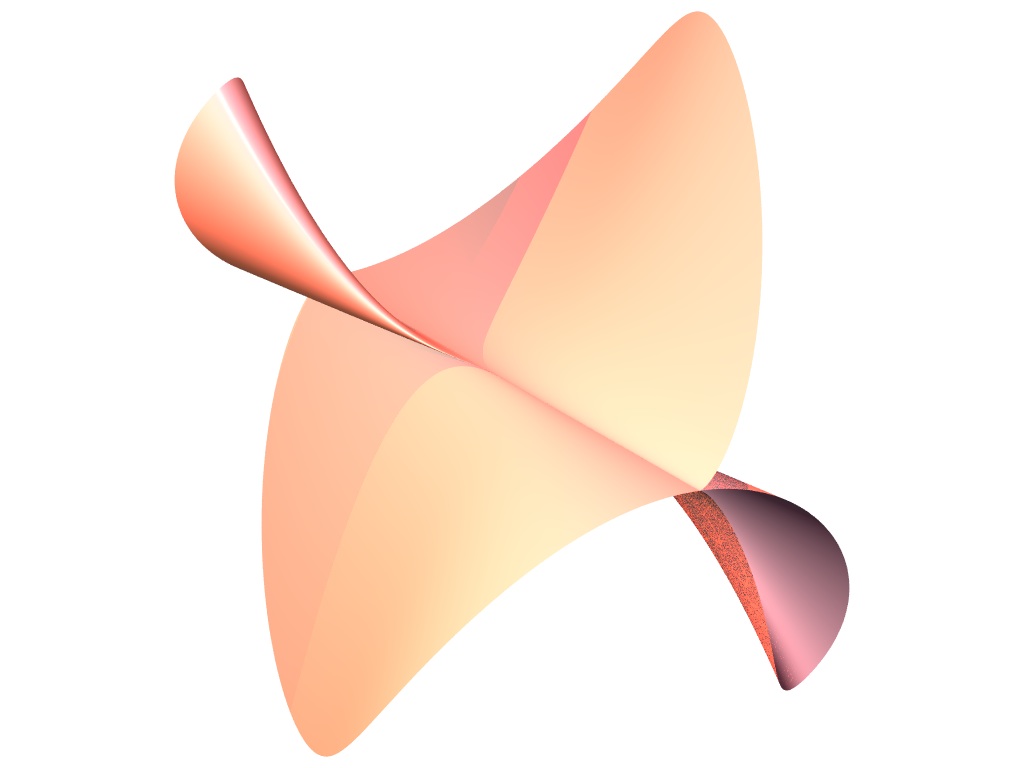}& 
\includegraphics[width=0.39 \textwidth]{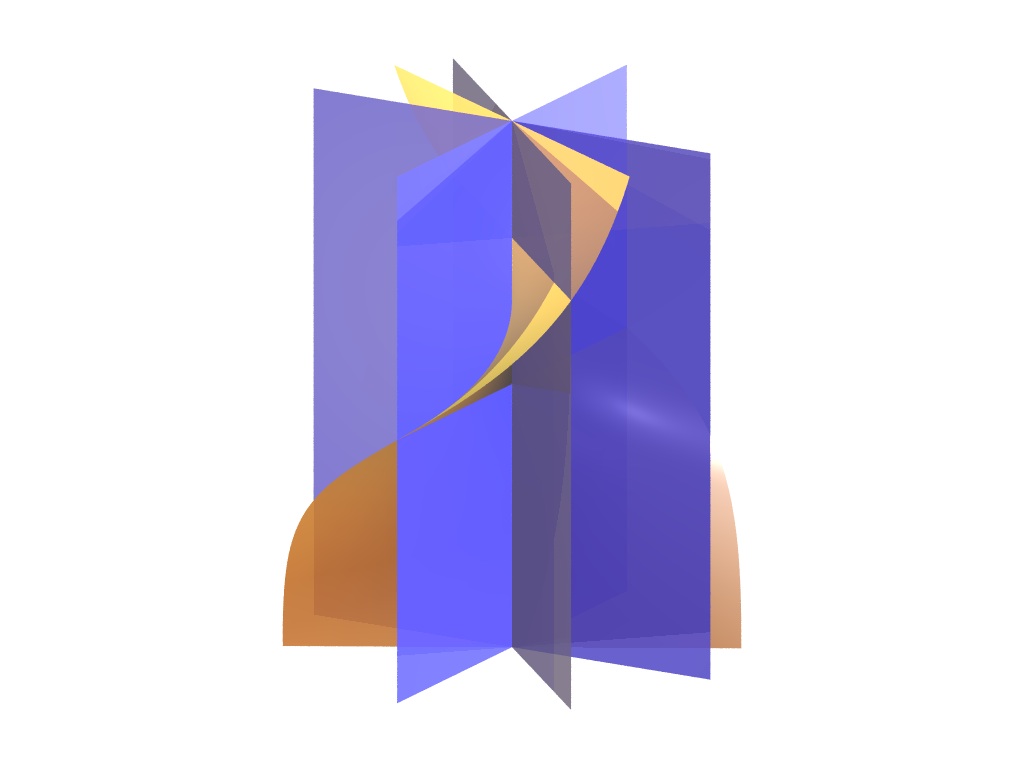}
\end{tabular}
\caption{ \label{Fig:freediv} Hessian deformation of $y^5+z^3$ (left)  and  The 4-lines (right).}
\end{figure}

\end{example}

So the right additional
requirement turns out to be radicality of the Jacobian ideal. Thus a purely algebraic criterion is obtained, which allows to determine whether a divisor has normal crossings at a point $p$, even
without knowing its decomposition into irreducible components.

\begin{theorem} \label{Thm:radikalJacobi}
Let $D$ be a divisor in a complex manifold $S$, $\dim S=n$.  Denote by $\pi: \widetilde D \ra D$ the normalization of $D$. Then the following are
equivalent: \\
(1) $D$ has normal crossings at a point $p$ in $D$. \\
(2) $D=\{h=0\}$ is free at $p$, the Jacobian ideal $J_{h,p}$ is radical and $(\widetilde D, \pi^{-1}(p))$ is smooth.
\end{theorem}

\begin{remark}
Condition (2) of the above theorem can also be phrased as: \\
(2') At $p \in D$ the Milnor algebra $\mc{O}_{\Sing D,p}$ is reduced and  either $0$ or Cohen--Macaulay of dimension $n-2$ and $\pi_*\mc{O}_{\widetilde D,p}$ is a regular ring. \\
Another equivalent formulation is: \\
(2'') At  $p \in D$, where $D=\{ h=0\}$, the Jacobian ideal $J_h$ is either equal to
$\mc{O}_{S,p}$ or it is radical, perfect with $\depth(J_h,\mc{O}_{S,p})=2$ and $\pi_*\mc{O}_{\widetilde D,p}$ is regular. 
\end{remark}

\begin{remark}
The condition $\widetilde D$ smooth is technical and only needed to apply Piene's Theorem (Thm.~\ref{Thm:Piene}) in our proof of Thm.~\ref{Thm:radikalJacobi}. In some special cases (see section
\ref{Sub:Special} and also Corollary \ref{Cor:ncomponentsweakly} and Lemma \ref{Lem:normalsmooth}) it can
be omitted. We do not know if this condition is necessary in general (cf. Remark \ref{Bem:CMGorensteinnormal}).
\end{remark}

About the proof of Theorem \ref{Thm:radikalJacobi}: the implication $(1) \Rightarrow (2)$ is a straightforward computation. We start with showing $(2) \Rightarrow (1)$ for some special cases, namely
for divisors in manifolds $S$ of dimension 2 (Prop.~\ref{Prop:radikaldim2}) and for $\Sing D$ Gorenstein (Prop.~\ref{Prop:Gorensteinsingular}). Note that in the section about free divisors and
logarithmic residue the theorem will be shown in the case where $D$ is a union of normal hypersurfaces.   For these cases, the assumption $\widetilde D$ smooth is not needed. The general proof of
$(2) \Rightarrow (1)$ occupies the rest of this section.

\subsection{Special cases of Theorem \ref{Thm:radikalJacobi}}  \label{Sub:Special}

\begin{lemma} \label{Prop:jacobian_singularlocus}
Let $D \subseteq S$ locally at a point $p$ be defined by an equation $h=0$. If $J_h$ is radical, then $h \in J_h$, which implies $\mc{O}_{\Sing
D,p}=\mc{O}_{S,p}/J_h$.
\end{lemma}

\begin{proof} 
One can show that  $h$ is contained in the ideal $\overline{J_h}$, the integral closure of $J_h$, see \cite{L-JT08}. 
It follows from the Theorem of Brian\c{c}on--Skoda that $\overline{J_h^{n}} \subseteq J_h$, see \cite{LT81}. Since
$(\overline{J_h})^n \subseteq \overline{J_h^n}$ (see for example \cite{L-JT08}), the $n$-th
power of $h$ is contained in $J_h$ and by our assumption $J_h$ already
contains $h$.
\end{proof}

\begin{remark}
The above lemma shows in particular that if $J_h$ is radical then also
$J_h=\overline{J_h}$.  
The blowup of $D$ in $J_h$ is the Nash blowup of $(D,p)$, see e.g. \cite{Nobile75}. It is an interesting question whether in the case of a radical Jacobian ideal this blowup is equal to the \emph{normalized Nash blowup} (see \cite[Section 3]{L-JT08}).
\end{remark}

\begin{proposition} \label{Prop:radikaldim2}
Let $\dim S=2$ and the divisor $D$ be defined at a point $p$ by a reduced $h \in \mc{O}_{S,p}$. Then $D$ has normal crossings at $p$ if and only if $D$ is free at $p$ and $D$ is either smooth or $J_h$
is radical of depth 2
on $\mc{O}_{S,p}$. 
\end{proposition}

\begin{proof}
If $D$ has normal crossings at $p$, then a simple computation shows the assertion. Conversely, suppose that $D$ is not smooth at $p$ and $J_h$ is radical of depth 2 on $\mc{O}_{S,p}$. Since $D$ is a
reduced curve, $J_h$ defines an isolated singularity, that is, $J_h$ is primary to the
maximal ideal $\mf{m}$ of $\calo_{S,p}$. Because $J_h$ is radical it follows that $J_h=\mf{m}$. This means that
$\mc{O}_{\Sing D,p}\cong \C$ at $p$. Now one
can use either a direct computation or apply the Theorem of Mather--Yau \cite{MY} for isolated singularities:
here it means that
$(D,p)$ is isomorphic to the normal crossings divisor $(N,p)$ (defined locally
at $p=(x_1,x_2)$ by the equation $\{ x_1 x_2=0\}$) if and only if their singular
loci are isomorphic. But $\mc{O}_{\Sing N,p}=\mc{O}_{S,p}/(x_1, x_2) \cong \C$ is clearly isomorphic to $\mc{O}_{\Sing D,p}$. 
Hence the assertion is shown.
\end{proof}

A particular class of Cohen--Macaulay rings are the so-called Gorenstein rings. In general, Gorenstein rings lie between complete intersections and Cohen--Macaulay rings. We prove here Thm.~\ref{Thm:radikalJacobi} for $\mc{O}_{\Sing D,p}$ Gorenstein. \\
Good references for definitions and properties regarding Gorenstein rings are \cite{Bass63,Eisenbud96,Matsumura86}.
In our situation, where the Jacobian ideal defining $\mc{O}_{\Sing D,p}$ has depth two on
$\mc{O}_{S,p}$, one sees that Gorenstein rings and complete intersection rings coincide, see \cite[Cor. 21.20]{Eisenbud96}:

\begin{theorem}[Serre] \label{Thm:Serre}
Let  $R$ be a regular local ring and $I \subseteq R$ an ideal with  $\depth(I,R)=2$. Then $R/I$ is
Gorenstein if and only if $I$ is generated by a regular sequence of length 2.
\end{theorem}

This leads to the statement of
\begin{proposition} \label{Prop:Gorensteinsingular}
Let $(S,D)$ be the pair of an $n$-dimensional complex manifold $S$ together with a divisor $D \subseteq S$ and $D=\{h=0\}$ at a point $p$. Suppose that $J_h$ is radical and
$\mc{O}_{\Sing D,p}$ is a Gorenstein ring of Krull-dimension $n-2$. Then $(\Sing D,p)$ is smooth and $D$  has locally at $p$ normal
crossings. 
\end{proposition}

First let us consider a possible counter-example to this proposition:

\begin{example} (The cusp) \label{Ex:cuspsingular} 
The (reduced) cusp curve in $\C^3$ is defined by $I=(x_{1}^3-x_2^2,x_3)$. Since $\mc{O}/I$ is a complete intersection ring, it is
Gorenstein. However, $\calo/I$ is clearly not
regular. In order that $I$ equals $J_h$ for some $h \in \mc{O}$ one must have $\D_{x_i} h=a_{i1} (x_1^3-x_2^2) + a_{i2}x_3$, for $i=1,2,3$. Now consider the $\C$-vector space $I/ \mf{m}I$. Since
$\mc{O}$ is a local ring, Nakayama's Lemma yields that $\overline{x_1^3-x_2^2}, \overline{x_3}$ form a basis of this vector space. From the Poincar\'e Lemma it follows
that  three
functions $f_1, f_2, f_3$ are partial derivatives $\D_{x_1}h, \D_{x_2}h, \D_{x_3}h$ if and only if $\D_{x_2}f_1=\D_{x_1}f_2, \D_{x_1}f_3=\D_{x_3}f_1, \D_{x_3}f_2=\D_{x_2}f_3$. Writing out these
conditions for  $a_{i1} (x_1^3-x_2^2) + a_{i2}x_3$ it follows that $a_{11}(0)=a_{21}(0)=a_{12}(0)=a_{22}(0)=0$. Hence modulo $\mf{m}$ the $\D_{x_i}h$ cannot be generated by $\overline{x_1^3-x_2^2},
\overline{x_3}$. By Nakayama's Lemma
this contradicts the fact that the $\D_{x_i}h$ also generate $I$. Thus $I$ cannot be the Jacobian ideal $J_h$ of some reduced $h$.
\end{example}

\begin{remark}
 One can construct surfaces in $\C^3$ with the cusp as singular locus by blowing down, see \cite{FaberHauser10}. However, then the Jacobian ideal will not be radical.
\end{remark}

\begin{lemma} \label{Lem:Gorensteinerzeuger}
Let $(S,D)$ be as before, with $\dim S=n$ and $D=\{h=0\}$ at a point $p=(x_1, \ldots, x_n)$. Suppose that the Jacobian ideal $J_h=(\D_{x_1}h, \ldots ,
\D_{x_n}h)$ is radical and $\mc{O}_{\Sing D,p}$ is Gorenstein of dimension $(n-2)$. Then $J_h$ can be generated by two derivatives $\D_{x_i}h, \D_{x_j}h$.
\end{lemma}

\begin{proof} 
Since $\mc{O}_{S,p}/J_h$ is Gorenstein, Thm.~\ref{Thm:Serre} yields that $J_h$ is generated by a regular sequence  $f,g$ in $\mf{m}$. Then there exists an $(n \times 2)$-matrix $A \in
M_{n,2}(\mc{O}_{S,p})$ such that
\begin{equation} \label{Equ:nakayamaJacobi}
 A  (f,g)^T= (\D_{x_1}h,\ldots, \D_{x_n} h)^T.
 \end{equation}
Consider the $\mc{O}_{S,p}/\mf{m}$-module $J_h / \mf{m} J_h$. 
Then, since $\mc{O}_{S,p}/\mf{m}=\C$,  the residue class $\overline{A}$ of the matrix $A$ of (\ref{Equ:nakayamaJacobi}) is in $M_{n,2}(\C)$. This means that we have obtained a solvable linear system of equations with
coefficients in $\C$.
Thus one sees that $f$ and $g$ are $\C$-linear combinations of two partial derivatives, wlog. of $\D_{x_1}h$ and $\D_{x_2}h$ modulo $\mf{m}J_h$.
This implies $J_h=(\D_{x_1}h, \D_{x_2}h) + \mf{m}J_h$ (as $\mc{O}_{S,p}$-modules). 
An application of Nakayama's Lemma shows the assertion.
\end{proof}

\begin{proof}[Proof of Prop.~\ref{Prop:Gorensteinsingular}]
From Lemma \ref{Lem:Gorensteinerzeuger} it follows that $J_h$ can be generated by two derivatives of $h$, wlog.  $J_h=(\D_{x_1}h, \D_{x_2}h)$. Hence one has $\D_{x_i}h =a_i (\D_{x_1}h) + b_i
(\D_{x_2}h)$, $a_i, b_i \in \mc{O}_{S,p}$, for $3 \leq i \leq n$. Consider vector fields $\delta_i=\D_{x_i}-a_i \D_{x_1} - b_i \D_{x_2}$ for $3 \leq i \leq n$. Then clearly one has $\delta_i(h)=0$.
Evaluation of these $n-2$ vector fields at $0$ shows that $\delta_3(0), \ldots, \delta_n(0)$ are $\C$-linearly independent vectors in $(S,p) \cong (\C^n,0)$.
Thus Rossi's Theorem can be applied (see \cite{Rossi63}): locally at $p$ the germ $(D,p)$ is isomorphic to $(D' \times \C^{n-2},(0,0))$, where $D'$ is locally contained in $\C^2$. Hence the problem
has been reduced
to $\dim_{\C}S=2$. Then Prop.~\ref{Prop:radikaldim2} shows that locally at $p$ the divisor $D$ is isomorphic to the union of two transversally intersecting hyperplanes. 
\end{proof}

\begin{remark}
Instead of using Rossi's Theorem in the above proof, we could use the argument in Lemma 2.3 of \cite{CNM96} and apply induction.
\end{remark}

\begin{remark}
One can also prove Proposition \ref{Prop:Gorensteinsingular} using Pellikaan's theory of the primitive ideal \cite{Pellikaan88}. Then one can show that $D$ is even analytically trivial along $\Sing D$, that is, $(D,p) \cong (D_0 \times \C^{n-2},(0,0))$, where $D_0$ is isomorphic to the fibre of $D$ at the origin defined by $h(x_1,x_2,0)$. This is carried out in detail in \cite{Faber11}.
\end{remark}

\subsection{General proof of Theorem \ref{Thm:radikalJacobi}}

The ideas to show the special cases do not lead to a proof in general. Therefore, our strategy to prove the general case is the following: \\
(i) If  $(D,p)=\bigcup_{i=1}^m (D_i,p)$ is free and a union of irreducible components and has radical Jacobian ideal, then we show that each $D_i$ is also free and has radical Jacobian ideal. \\
(ii) If $D$ is free, irreducible, has radical Jacobian ideal at $p$ and the normalization $\widetilde D$ is smooth, then $D$ is already smooth at $p$. \\
(iii) A free divisor $D$, which is a union of smooth irreducible hypersurfaces and has a radical Jacobian ideal, is already a normal crossing divisor. \\

In order to obtain (i) and (iii) we will use the results from \cite{Faber12}: therefore so-called splayed divisors are introduced. A splayed divisor $D$ is a union of transversally meeting
hypersurfaces that are possibly singular (see definition below). A divisor $(D,p)=\bigcup_{i=1}^m (D_i,p)$ with radical Jacobian ideal is splayed (Prop.~\ref{Prop:produktfreireduktion}). Then claim (iii) follows from Corollary
\ref{Cor:smoothcomp}. Note that (iii) also follows from Lemma \ref{Lem:normalsmooth}, where the logarithmic residue is used.
Claim (ii) follows from Piene's Theorem (Thm.~\ref{Thm:Piene}).

\begin{definition} \label{Def:produkt}
Let $D$ be a divisor in a complex manifold $S$, $\dim S=n$. The
divisor $D$ is called \emph{splayed} at a point $p \in S$ (or $(D,p)$ is \emph{splayed}) if one can find coordinates $(x_1, \ldots, x_n)$ at $p$ such that $(D,p)=(D_1,p) \cup (D_2,p)$ is defined by
$$h(x)=h_1(x_1, \ldots, x_k) h_2(x_{k+1}, \ldots, x_n),$$
$1\leq k \leq n-1$, where $h_i$ is the defining reduced equation of $D_i$. Note that the $h_i$ are not necessarily irreducible. The
\emph{splayed components} $D_1$ and $D_2$ are not unique. Splayed means that $D$ is the union of two products: since $h_1$ is independent of $x_{k+1}, \ldots, x_n$, the divisor $D_1$ is locally at $p$
isomorphic to a product $(D'_1,0) \times (\C^{n-k},0)$, where $(D'_1,0) \subseteq (\C^k,0)$ (and similar for $D_2$). 
\end{definition}

\begin{proposition} \label{Prop:produktfreireduktion}
Let $D=D_1 \cup D_2$ be a divisor in an $n$ dimensional complex manifold $S$ and let $D$, $D_1$ and $D_2$ at a point $p \in S$ be defined by the equations  $gh$,  $g$ and $h$, respectively. Suppose
that $J_{gh}$ is radical. Then $D$ is splayed and $J_h$ and
$J_g$ are also radical. If moreover $D$ is free at $p$ then also $D_1$ and $D_2$
are free at $p$.
\end{proposition}

\begin{proof}
See \cite{Faber12}.
\end{proof}

From this follows (using induction on $n$, see \cite{Faber12})

\begin{corollary}  \label{Cor:smoothcomp}
Let $(S,D)$ be a complex manifold, $\dim S=n$, together with a divisor $D \subseteq S$ and suppose that locally at a point $p \in S$ the divisor $(D,p)$ has the  decomposition into irreducible
components $\bigcup_{i=1}^m
(D_i,p)$ such that each $(D_i,p)$ is smooth. Let the corresponding equation of $D$ at $p$ be $h=h_1, \ldots, h_m$.
If $D$ is free at $p$ and $J_h=\sqrt{J_h}$ then $D$ has normal crossings at $p$.
\end{corollary}

In order to state Piene's Theorem below, we need some properties of the normalization of $D$, in particular of the conductor ideal.  \\
Let $(X,x)$ be the germ of an equidimensional analytic space with normalization $\pi: \widetilde X \ra X$. Then the \emph{conductor ideal} $C_{X,x}$ at  $x$ is the largest ideal that is an ideal in
$\mc{O}_{X,x}$ as well as in $\mc{O}_{\widetilde X,x}$  (we write $C_X$ if there is no danger of confusion regarding the point $x$). Here note that $\mc{O}_{\widetilde X,x}$ is canonically isomorphic
to $\pi_*\mc{O}_{\widetilde X,x}$ and also to the ring of weakly holomorphic functions on $X$, see e.g. \cite{dJP}. Alternatively, the conductor $C_{X,x}$ can be defined as the
ideal quotient $(\mc{O}_{X,x} : \pi_*\mc{O}_{\widetilde X,x})=\{ f \in \mc{O}_{X,x}: f \pi_*\mc{O}_{\widetilde X,x} \subseteq \mc{O}_{X,x} \}$ or as
$\mathrm{Hom}_{\mc{O}_{X,x}}(\pi_*\mc{O}_{\widetilde X,x},\mc{O}_{X,x})$.
Note that $C_X$ is a coherent sheaf of ideals over $\mc{O}_{X}$.

\begin{theorem}[Piene's Theorem]  \label{Thm:Piene}
Let $X$ be a locally complete intersection variety of dimension $s$ over an algebraically closed field $k$. Let $f: Z \ra X$ be a desingularization of $X$ and denote by $I_f=F^0_{Z} (\Omega^1_{Z /
X})$ the ramification ideal of $f$ in $\mc{O}_Z$ and by $J_X$ the ideal $F^s_X(\Omega^1_{X/k})$. Suppose that $f$ is finite. Then there is an equality of ideals
$$ J_X \mc{O}_Z = I_f C_X \mc{O}_Z.$$
\end{theorem}

\begin{proof}
See Theorem 1 and Corollary 1 of \cite{Piene79}.
\end{proof}

\begin{remark}  \label{Bem:Piene}
(1) The above theorem also holds in the analytic case since all constructions in the proof of Theorem 1 of \cite{Piene79} also work, cf. \cite{MondPellikaan,Aleksandrov90,GrangerSchulze11}.   \\
(2) The ideal $J_X$ is sometimes also called ``Jacobian ideal of $X$''. We need the above theorem in the case where $X$ is a divisor $D$ in a complex manifold $S$ defined locally at a point $p$ by
$\{h=0\}$. Then $J_D$ is simply the ideal $J_h$ in $\mc{O}_{D,p}$ (resp. the ideal $((h)+J_h) \subseteq \mc{O}_{S,p}$) defining the singular locus $(\Sing D,p)$. Clearly, $D$ is at $p$ a
complete intersection. 
\end{remark}

\begin{lemma}  \label{Lem:Jacobiconductor}
Let $(S,D)$ be a complex manifold, $\dim S=n$, together with a divisor $D \subseteq S$ and suppose that $D=\{ h=0\}$ is free at $p$ and that $J_h= \sqrt{J_h}$. Then the Jacobian ideal equals the
conductor of the normalization, that is $J_h=C_{D,p}$.
\end{lemma}

\begin{proof}
The inclusion $J_h \subseteq C_{D,p}$ always holds. Since in case of a free divisor the singular locus is non-normal, it follows that $\mathrm{supp}(\calo_{D,p}/J_h)=\mathrm{supp}(\calo_{D,p}/C_{D,p})$. Because $J_h$ is radical, it is actually equal to $C_{D,p}$. 
\end{proof}

\begin{lemma}  \label{Lem:unramifiedisomorphism}
 Let $\pi: \widetilde D \ra D$ be the normalization of an irreducible $(D,p)$ and suppose that $\pi$ is unramified, that is, $\Omega^1_{\widetilde D / D}=0$. Then $\pi$ is an isomorphism.
\end{lemma}

\begin{proof} See e.g. Lemma 4.1 of \cite{GrangerSchulze11}.
\end{proof}

\begin{proof}[Proof of Theorem \ref{Thm:radikalJacobi}]
First let us suppose that $D$ is singular and has normal crossings at $p \in \Sing D$. Then we can assume that $D=\bigcup_{i=1}^m (D_i,p)$ is given by the equation $h=x_1 \cdots x_m$, $1 < m \leq n$ where each $x_i$
corresponds to an irreducible component $D_i$ passing through $p$. Then 
$$J_h=\sum_{i=1}^m (x_1 \cdots \hat x_i \cdots x_m).$$
The ideal $J_h$ is the ideal generated by the maximal minors of the $(n-1) \times n$ matrix with rows $(x_1, 0 , \ldots,0, -x_i,0, \ldots, 0)$ for $i=2, \ldots, n$. Therefore $\calo / J_h$ is Cohen--Macaulay of dimension $n-2$ by the Theorem of Hilbert--Burch. 
Using facts about primary decomposition of monomial ideals, see e.g.\cite{HostenSmith02}, it follows that
$$J_h=\bigcap_{1 \leq i < j \leq m} (x_i, x_j),$$
which is clearly radical.  The normalization of a normal crossing divisor $D=\bigcup_{i=1}^m D_i$ is smooth since it is the
disjoint union of the smooth components $D_i$. \\
Conversely, suppose that $J_h=\sqrt{J_h}$ and $\mc{O}_{\Sing D,p}$ is Cohen--Macaulay of
dimension $(n-2)$  and moreover that $\mc{O}_{\widetilde D, \pi^{-1}(p)}$ is regular (denote by $\pi: \widetilde D \ra D$ the normalization morphism, resp. by $\pi_i: \widetilde D_i \ra D_i$ the
normalization of $D_i$). Note here that always $\mc{O}_{\widetilde D}=\bigoplus_{i=1}^m \mc{O}_{\widetilde D_i}$ holds, see \cite{dJP}.
Prop.~\ref{Prop:produktfreireduktion} implies that each $D_i$ is free at $p$ and has a radical Jacobian ideal $J_{h_i}$.  
By  our hypothesis, Piene's Theorem (Thm.~\ref{Thm:Piene}) and remark \ref{Bem:Piene} yield the equality of ideals
$$ C_{D_i}  I _{\pi_i} \mc{O}_{\widetilde D_i,p} =J_{h_i} \mc{O}_{\widetilde D_i,p}.$$
Since by Lemma \ref{Lem:Jacobiconductor}  one has $J_{h_i}=C_{D_i,p}$ in $\mc{O}_{D_i,p}$, it follows, using  Nakayama's Lemma, that
$I_{\pi_i}=\mc{O}_{\widetilde D_i,p}$.  Hence $\Omega^1_{\widetilde D_i / D_i}=0$.
By Lemma \ref{Lem:unramifiedisomorphism}, each $\pi_i$ is an isomorphism and thus $D_i$ is already normal at $p$. By definition, the only free divisor that is normal is the smooth divisor, so it
follows that each $D_i$  is  smooth at $p$.
For $(D,p)=\bigcup_{i=1}^m(D_i,p)$ this means that we are in the situation of Corollary \ref{Cor:smoothcomp} and the assertion follows. 
\end{proof}

\begin{remark}
We can also give a different proof of $(2) \Rightarrow (1)$ of Thm.~\ref{Thm:radikalJacobi} using the characterization of normal crossings by the logarithmic residue of
Prop.~\ref{Thm:irreduzibel}:  let $(D,p)=\bigcup^{m}_{i=1}(D_i,p)$ be the decomposition into irreducible components and suppose that $J_h=\sqrt{J_h}$. Then the singular locus of the singular locus, $\Sing (\Sing D)$,  is of
dimension less than or equal to $(n-3)$. By Proposition \ref{Prop:Gorensteinsingular}, $D$ has normal crossings at smooth points of $\Sing D$. Hence $D$ has normal crossings in codimension 1. From a result of Saito \cite[Lemma 2.13]{Saito80} it follows that the logarithmic residue is holomorphic on the normalization, that is, $\rho(\Omega^1_{S}(\log D))=\pi_*\mc{O}_{\widetilde D}$. Then Prop.~\ref{Thm:irreduzibel} shows that
$D$ is a normal crossing divisor. \\
\end{remark}

\begin{remark} \label{Bem:CMGorensteinnormal}
We do not know whether the condition on the normalization of $D$ in Theorem \ref{Thm:radikalJacobi} is necessary. If $(D,p)$ is free and has a radical Jacobian ideal, then by Lemma
\ref{Lem:basisweaklyholom} the normalization $(\widetilde D,\pi^{-1}(p))$ is  Cohen--Macaulay. Here the question is if Piene's Theorem about the equality of the ideals holds in a more general context
or if we can find an alternative argument to show that the normalization $\pi: \widetilde D \ra D$ is unramified.
\end{remark}

\begin{Conj}  \label{Conj:gorensteinnormal}
Let $D \subseteq S$ be a divisor in a complex manifold $S$ that is locally at a point $p$ given by $h=0$ and denote by $\pi:\widetilde D \ra D$ its normalization. Suppose that $D$ is free at $p$ and
that $J_h=\sqrt{J_h}$. Then the normalization $\widetilde D$ of $D$ is already smooth at $\pi^{-1}(p)$.
\end{Conj}

This conjecture is supported by the results of Granger and Schulze about the dual logarithmic residue, see \cite{GrangerSchulze11}. 

\subsection{Radical Jacobian ideals} \label{Sub:radicalJacobiallg}

 Let $D$ be a divisor in a smooth  $n$-dimensional manifold $S$ that is locally at a point $p$ given by  $h \in \mc{O}_{S,p}$. Suppose that $J_h$ is radical. Which ideals $I \subseteq \mc{O}_{S,p}$ can be such radical Jacobian ideals $J_h$?
More precisely: given a radical ideal $I \subseteq \mc{O}_{S,p}$, when does there exist a divisor $(D,p)=\{ h=0\}$ such that $I=J_h$? \\

If $I$ is a complete intersection then in Proposition \ref{Prop:jacobivollstdsallg} it is shown that $I$ is a Jacobian ideal if and only if it defines a smooth variety. Apart from that, the case of
$\dim S=2$ was treated in Proposition \ref{Prop:radikaldim2}. For $\dim S=3$ it also easily follows that for a divisor $D=\{ h=0\}$ with $J_h=\sqrt{J_h} \neq (1)$ and smooth normalization 
one of the two cases occurs: \\
(i) $\depth(J_h,\mc{O}_{S,p})=3$ and $(D,p)$ is an $A_1$-singularity. \\
(ii) $\depth(J_h, \mc{O}_{S,p})=2$ and $D$ has normal crossings at $p$. \\
Here (i) directly follows from Prop.~\ref{Prop:jacobivollstdsallg}. To prove (ii) one uses that  a reduced one-dimensional local ring is Cohen--Macaulay, see e.g. \cite{Eisenbud96}, and Theorem
\ref{Thm:radikalJacobi}.

\begin{Qu}
Does there exist a surface $(D,p) \subseteq (\C^3,p)$ such that $(D,p)$ is free and $J_h=\sqrt{J_h}\neq (1)$ but $(\widetilde D,\tilde p)$ is not smooth? 
\end{Qu}

For $\dim S \geq 4$ the situation is more involved, we split it into two parts. 

\subsubsection{Codimension 1 singular locus}

\begin{example} \label{Ex:singequidim}
 Consider a manifold $(S,p) \cong (\C^4,0)$ with coordinates $p=(x,y,z,w)$. Then the ideal $I=(x,y) \cap (z,w)=(xz,xw,yz,yw) \subseteq \mc{O}_{S,p}$ is radical and defines an
equidimensional 2-dimensional analytic space germ $(Z,p)$. One can show that $\mc{O}_{S,p}/I$ is not Cohen--Macaulay, which implies that $I$ is not a complete
intersection. By computation we show that there does not exist an $h \in \mc{O}_{S,p}$ such that $I=J_h$: first note that $I$ is the Jacobian ideal of a divisor defined by some $h \in \mc{O}_{S,p}$ if
and only if there
exists a matrix $A \in GL_4(\mc{O}_{S,p})$ such that
\begin{equation} \label{Equ:jacob}
(\D_{x}h, \D_y h, \D_z h, \D_w h)^T= A \underline{f}^T,
\end{equation}
where $\underline{f}$ is the vector $(f_1, \ldots, f_4):=(xz,yz,xw,yw)$. 
Hence the matrix $A(0)$ has to be in $GL_4(\C)$. The partial derivatives of $h$ have to satisfy six equations, namely
\begin{align*}  \D_{xy}h &  = \D_{yx}h, \D_{xz}h= \D_{zx} h, \D_{xw} h = \D_{wx}h, \\ 
\D_{yz}h & =\D_{zy}h, \D_{yw}h=\D_{wy}h, \D_{zw}h=\D_{wz}h.
\end{align*}
An explicit comparison of the order zero terms of (\ref{Equ:jacob}) plugged into these relations shows that in the matrix $A(0)$ the first
row is zero, which means that $A \not \in GL_4(\mc{O}_{S,p})$. 

\end{example}

\begin{Conj} 
Let $D \subseteq S$ be a divisor defined at a point $p$ by $h \in \mc{O}_{S,p}$. If $J_h$ is radical, of depth 2 on $\mc{O}_{S,p}$  and equidimensional, then $\mc{O}_{S,p}/J_h$ is already
Cohen--Macaulay. In other words: we conjecture that if a divisor that has locally at a point $p \in S$ an  equidimensional radical Jacobian ideal of depth 2  is already free at $p$.
\end{Conj}

If $J_h$ is not equidimensional, the only thing we can say is that it is the intersection of some prime ideals whose minimal height is 2. 

\begin{example}
(The Jacobian ideal can be of height 2 and radical but it may not be equidimensional) Consider
$S=\C^5$ at the origin with coordinates $(x,y,z,s,t)$. Let the divisor $D$ be locally defined by $h=(x^2+y^2+z^2)(s^2-t^2) \in \mc{O}=\C\{x,y,z,s,t\}$. Note that $D$ is splayed and the union of a
normal crossing divisor and a cone. The Jacobian ideal $J_h$ is radical, its height is 2 and it has the prime decomposition
$$(x,y,z) \cap (s-t,x^2+y^2+z^2) \cap (s,t) \cap (s+t, x^2+y^2+z^2).$$
The ideal $J_h$ is not unmixed and hence $\mc{O}_{\Sing D}=\mc{O}/J_h$ is not Cohen--Macaulay.
\end{example}

\begin{Qu} 
 Suppose that $J_h$ of a divisor $D$ is radical and of height 2 but not equidimensional. Which $J_h$ are possible?
\end{Qu}

\subsubsection{Higher codimensional singular locus}
In this case the divisor $D$ has to be irreducible and normal.

\begin{proposition} \label{Prop:jacobivollstdsallg}
 Let $J_h=\sqrt{J_h}$ be the Jacobian ideal of the divisor $D \subseteq S$ and denote by $\Sing D$ its singular locus with associated ring $\mc{O}_{\Sing D,p}=\mc{O}_{S,p} / J_h$ at $p$. Suppose that
$\mathrm{codim}_p( \Sing D, S)=k$ and that $(\Sing D,p)$ is a complete intersection. Then $D$ is isomorphic to $\{ x_1^2 + \cdots+ x_k^2=0\}$, that is, $D$ has locally along $\Sing D$ an
$A_1$-singularity. \\
In particular, it follows that $(\Sing D,p)$ is a complete intersection if and only if $(\Sing D,p)$ is smooth and thus $D$ is isomorphic to $\{ x_1^2 + \cdots+ x_k^2=0\}$.
\end{proposition}

\begin{proof}
 The proof is  similar to the free divisor case, see Prop.~\ref{Prop:Gorensteinsingular}. Since $(\Sing D,p)$ is a complete intersection, there exist $f_1, \ldots, f_k \in \mc{O}_{s,p}$ such that
$J_h=(f_1, \ldots ,f_k)$. Since the $f_i$ generate $J_h$, there is an $n \times k$ matrix $A$ with entries in $\mc{O}_{S,p}$ such that 
$$ A (f_1, \ldots, f_k)^T=(\D_{x_1}h, \ldots, \D_{x_n}h)^T. $$ 
Similar to Lemma \ref{Lem:Gorensteinerzeuger} consider the $\mc{O}/\mf{m}=\C$ module $J_h/ \mf{m} J_h$.
Hence wlog. one may assume that $J_h=(\D_{x_1}h, \ldots, \D_{x_k}h)$.
Similar to the proof of Prop.~\ref{Prop:Gorensteinsingular},
$D$   may be considered in $(\C^k,0)$ and defined by $h^*(x_1, \ldots, x_k)=h(x_1, \ldots, x_k,0)$. Then since $J_{h^*}$ defines a complete intersection of codimension $k$ in $(\C^k,0)$, it defines an
isolated singularity. 
Like in the 2-dimensional case (Prop.~ \ref{Prop:radikaldim2}) we find that $D$ is locally isomorphic to $\{ x_1^2+ \cdots + x_k^2=0\}$. 
\end{proof}

It is not clear how to treat non-complete intersection radical Jacobian ideals of height $k$, $3 \leq k <n$ in $\mc{O}$, we do not even know if there exist divisors $D=\{h=0\}$ such that
$J_h=\sqrt{J_h}$ of height $k \geq 3$ is not a complete intersection. For example, one can show in a similar way like in example \ref{Ex:singequidim} that the equidimensional radical ideal
$I=(x_1,x_2,x_3) \cap (x_4,x_5,x_6) \subseteq \mc{O}=\C\{x_1, \ldots, x_6\}$ cannot be the Jacobian ideal of some $h \in \calo$.

The propositions and examples above motivate the following 

\begin{Conj}
 Let $D$ be a divisor in a complex manifold $S$, defined locally at a point $p$ by a reduced $h \in \mc{O}_{S,p}$. Suppose that the Jacobian ideal $J_h$ is radical, equidimensional and of depth $\geq
3$ on $\mc{O}_{S,p}$. Then the variety $\Sing
D$ with coordinate ring $\mc{O} / J_h$ is a complete intersection, that is, $\Sing D$ is Cohen--Macaulay and must even be smooth by Prop.~\ref{Prop:jacobivollstdsallg}.
\end{Conj}

\section{Normal crossings and logarithmic differential forms and vector fields}   \label{Sec:logdiff}

Usually, free divisors are introduced via logarithmic differential forms and vector fields.  These will be introduced here and also the logarithmic residue, which will be needed in this and the next section. The corresponding theory was developed by K.~Saito in \cite{Saito80}, where also proofs for most of our assertions can be found.\\
In this section we give a characterization of a normal crossing divisor in terms of generators of its module of logarithmic differential forms resp.~vector fields (Prop.~\ref{Thm:ncequivclosedforms}).
Namely, a divisor $D \subseteq S$ has normal crossings at a point $p$ if and only if $\Omega^1_{S,p}(\log D)$ is a free $\mc{O}_{S,p}$-module and has a basis
of closed forms or if and only if $\Der_{S,p}(\log D)$ is a free
$\mc{O}_{S,p}$-module and has a basis of commuting vector fields (this means that  there
exist logarithmic derivations $\delta_1, \ldots, \delta_n$ such that the $\delta_i$ form a basis of $\Der_{S,p}(\log D)$ and
$[\delta_i,\delta_j]=0$ for all $i,j=1, \ldots, n$). We remark that we only show the \emph{existence} of bases with these properties of
$\Omega^1_{S,p}(\log D)$ and $\Der_{S,p}(\log D)$ in case $D$ has normal crossings at $p$. We do not have a procedure
to \emph{construct} such bases. Thus, strictly considered, Prop.~\ref{Thm:ncequivclosedforms} does not satisfy our requirements on an effective algebraic criterion for normal crossings. \\


Let $D$ be a divisor in $S$ defined at $p$ by $D=\{h=0\}$. A \emph{logarithmic vector field} (or \emph{logarithmic derivation}) (along $D$) is a holomorphic vector field  on $S$, that is, an element
of $\Der_{S}$, satisfying
one of the two equivalent
conditions: \\
(i) For any smooth point $p$ of $D$, the vector $\delta(p)$ of $p$ is tangent to $D$, \\
(ii) For any point $p$, where $(D,p)$ is given by $h=0$, the germ $\delta(h)$ is contained in the ideal
$(h)$ of  $\mc{O}_{S,p}$. \\
The module of germs of logarithmic derivations (along $D$) at $p$ is denoted by 
\[  \Der_{S,p}(\log D)=\{  \delta: \delta \in \Der_{S,p} \text{ such that }\delta h \in (h)  \}. \]
These modules are the stalks at points $p$ of the sheaf $\Der_S(\log D)$ of $\mc{O}_{S}$-modules. 
Similarly we define logarithmic differential forms: a meromorphic $q$-form $\omega$ is logarithmic (along
$D$) at a point $p$ if $\omega h$ and $hd\omega$ are holomorphic in an
open neighbourhood around $p$. We denote
\[\Omega^q_{S,p}(\log D)= \{ \omega: \omega \text{ germ of a logarithmic $q$-form at $p$} \}.\] 
One can show that $\Der_{S,p}(\log D)$ and $\Omega^1_{S,p}(\log D)$ are reflexive $\mc{O}_{S,p}$-modules (see \cite{Saito80}). By a Theorem of Aleksandrov \cite{Aleksandrov90}, $(D,p)$ is free if and only
$\Der_{S,p}(\log D)$ resp. $\Omega_{S,p}^1(\log D)$ is a free $\calo_{S,p}$-module. The following theorem makes it possible to test whether $D$ is free (cf.~\cite[Thm.~1.8]{Saito80}): 

\begin{theorem}[Saito's criterion] \label{Thm:Saito}
Let $(S,D)$, $p$ and $h$ be as above. Then
$\Omega^1_{S,p}(\log D)$ is a free $\mc{O}_{S,p}$-module if and only if one has $\bigwedge^n \Omega^1_{S,p} (\log D)= \frac{1}{h}\Omega^n_{S,p}(\log D)$. This means that there exist $n$ elements
$\omega_i \in
\Omega^1_{S,p}(\log D)$ such that 
\[ \omega_1 \wedge \ldots \wedge \omega_n= u \frac{dx_1 \wedge \ldots \wedge dx_n}{h},\]
where $u$ is a unit in $\mc{O}_{S,p}$, i.e., $u \in \mc{O}_{S,p}^*$.
Then the $\omega_1, \ldots, \omega_n$ form an $\mc{O}_{S,p}$-basis for $\Omega^1_{S,p}(\log D)$ and one can write
$$\Omega^q_{S,p}(\log D)=\sum_{i_1 < \cdots < i_q} \mc{O}_{S,p} \ \omega_{i_1} \wedge \cdots \wedge \omega_{i_q},$$
for all $q=1, \ldots, n$. \\
A similar statement holds for $\Der_{S,p}(\log D)$.
\end{theorem}

In the following the so-called logarithmic residue will be used.  It is a tool to study the structure of the module
of logarithmic differential forms along $D$. It is tightly connected to the normalization of $D$. Locally, the residue of
$\Omega^1_{S,p}(\log D)$ is contained in the ring of meromorphic functions
$\mc{M}_{D,p}$ on $D$. In some way it measures how far away a logarithmic $q$-form is
from being holomorphic.\\
Historically, the study of residues of differential forms was initiated by A.~Cauchy in 1825: he considered residues of holomorphic functions in one variable.
Later, in 1887, H.~Poincar{\'e} introduced the notion of the residue of a
rational 2-form in $\C^2$. This was generalized by G.~de Rham and J.~Leray to
the class of $d$-closed meromorphic $q$-forms with poles of first order along a
smooth divisor. The modern algebraic treatment of residues in duality theory is
due to Leray and Grothendieck, see for example \cite{RD}.
We will study the  \emph{logarithmic} residue as
introduced by K.~Saito. More about the logarithmic residue can be found in \cite{AleksandrovTsikh01,Aleksandrov05,GrangerSchulze11}.  \\

 Let $S$ be an $n$-dimensional complex manifold and $D$ a divisor in $S$
given locally at a point $p \in S$ by a reduced equation $h \in \mc{O}_{S,p}$ and denote by $\pi: \widetilde D \ra D$ the normalization of $D$. Let $\mc{O}_{D}$ and $\mc{M}_{D}$ (resp.
$\mc{O}_{\widetilde D}$ and $\mc{M}_{\widetilde D}$) be the sheaves of germs of holomorphic and meromorphic functions on $D$ (resp. $\widetilde D$). Further denote by $\Omega^q_{D}$ (resp.
$\Omega^q_{\widetilde D}$) the sheaf of germs of holomorphic $q$-forms on $D$ (resp. $\widetilde D$). One has $\mc{O}_{D,p}=\mc{O}_{S,p}/ (h) \mc{O}_{S,p}$ and
$\Omega^q_{D,p}=\Omega^q_{S,p}/(h\Omega^q_{S,p}+dh\wedge\Omega^{q-1}_{S,p})$ and also $\mc{M}_{D}\otimes_{\mc{O}_{D}} \Omega^q_{D}=\pi_{*}(\mc{M}_{\widetilde D} \otimes_{\mc{O}_{\widetilde D}}
\Omega^q_{\widetilde D})$. In particular for $q=0$ we have $\pi_{*}(\mc{M}_{\widetilde D})=\mc{M}_{D}$ since $\pi$ is birational.

\begin{definition}
Let $(S,D)$, $p$ and $h$ be defined as usual. Let $\omega$ be any element in 
$\Omega^q_{S,p}(\log D)$. Then one can find a
presentation (see \cite[1.1]{Saito80}) 
\[g\omega=\frac{dh}{h}\wedge \xi + \eta, \]
with $g$ holomorphic and $\dim \mc{O}_{D,p}/(g) \mc{O}_{D,p} \leq n-2$, $\xi
\in \Omega^{q-1}_{S,p}$ and $\eta \in \Omega^q_{S,p}$. The \emph{residue
homomorphism} $\rho$ is defined as the $\mc{O}_{S,p}$-linear homomorphism of sheaves 
\begin{align*} \rho:\Omega^q_{S}(\log D)  & \lra  \mc{M}_{D}
\otimes_{\mc{O}_{D}}\Omega^{q-1}_{D} \\
  \omega & \longmapsto  \rho(\omega)=\frac{\xi}{g}.
\end{align*}
We often call $\rho(\Omega^1_{S,p}(\log D)) \subseteq \mc{M}_{D,p}$ the \emph{logarithmic residue} (of $D$ at $p$).
\end{definition}

One can show that the residue homomorphism $\rho$ is well defined, see \cite[2.4]{Saito80}. For an $\omega \in \Omega^q_{S,p}(\log D)$, $\rho(\omega)$ is $0$ on $D$ if and only if $\omega \in \Omega^q_{S,p}$. Moreover, $\rho(\Omega_{S}^q (\log D))$ is an $\mc{O}_{\widetilde D}$-coherent submodule of $\mc{M}_{\widetilde D} \otimes \Omega^{q-1}_{\widetilde D}$. In particular, the logarithmic residue $\rho(\Omega^1_{S}(\log D))$ contains $\pi_{*}\mc{O}_{\widetilde D}$, the ring of weakly holomorphic functions on $D$. \\

For the characterization of normal crossings we start with the following theorem \cite[Thm.~2.9]{Saito80}.

\begin{theorem}[Saito's Theorem] \label{Thm:closednc}
Let $(S,D)$ be a pair of a complex $n$-dimensional manifold and a divisor $D \subseteq S$. Suppose that locally at a point $p$
the divisor $D$ decomposes into irreducible components $(D,p)=(D_{1},p) \cup \ldots
\cup (D_{m},p)$. Let $h=h_{1} \cdots h_{m}$ be the corresponding decomposition of
the local equation of $D$, each irreducible factor $h_i$ corresponding to $D_i$. Then the following
conditions are equivalent: \\
(i) $\Omega^1_{S,p} (\log D)=\sum_{i=1}^m \mc{O}_{S,p}\frac{dh_{i}}{h_{i}} + \Omega_{S,p}^1$. \\
(ii) $\Omega^1_{S,p} (\log D)$ is generated by closed forms. \\
(iii) $\rho(\Omega^1_{S,p} (\log D))=\bigoplus_{i=1}^m \mc{O}_{D_{i},p}$. \\
(iv) (a) For each $i=1, \ldots, m$ the component $D_{i}$ is normal (i.e.,  \\
\phantom{(iv) (a)} 
$\dim \Sing D_{i} \leq  n-3$), \\
\phantom{(iv)} (b) $D_{i}$ intersects $D_{j}$ transversally for $i \neq j$ and $i,j = 1, \ldots, m$, \\
\phantom{(iv)} (c) $\dim(D_{i} \cap D_{j} \cap D_{k}) \leq n-3$ for all
$i,j,k$ distinct and 
$i,j,k =1, \ldots, m$.
\end{theorem}

\begin{example}  \label{Ex:tuellefrei}
Let $D$ be the divisor in $\C^3$ defined by $h=xz(x+z-y^2)$. This divisor is
called \emph{T\"ulle} and is studied in more detail in \cite{FaberHauser10}.
T\"ulle consists of three components, which are smooth, intersect pairwise
transversally and whose triple intersection is a point, see fig.~\ref{Fig:tuelle}. Thus it fulfills the
assumption (iv) of Theorem \ref{Thm:closednc}. The local ring $\mc{O}_{\Sing D,0}$ defining the singular locus $(\Sing D,0)$ is
not Cohen--Macaulay. Note that $D$ has  normal crossings outside the origin, where it is not even free. \\
\end{example} 

\begin{figure}[!h]
\begin{center}
\includegraphics[width=0.4 \textwidth]{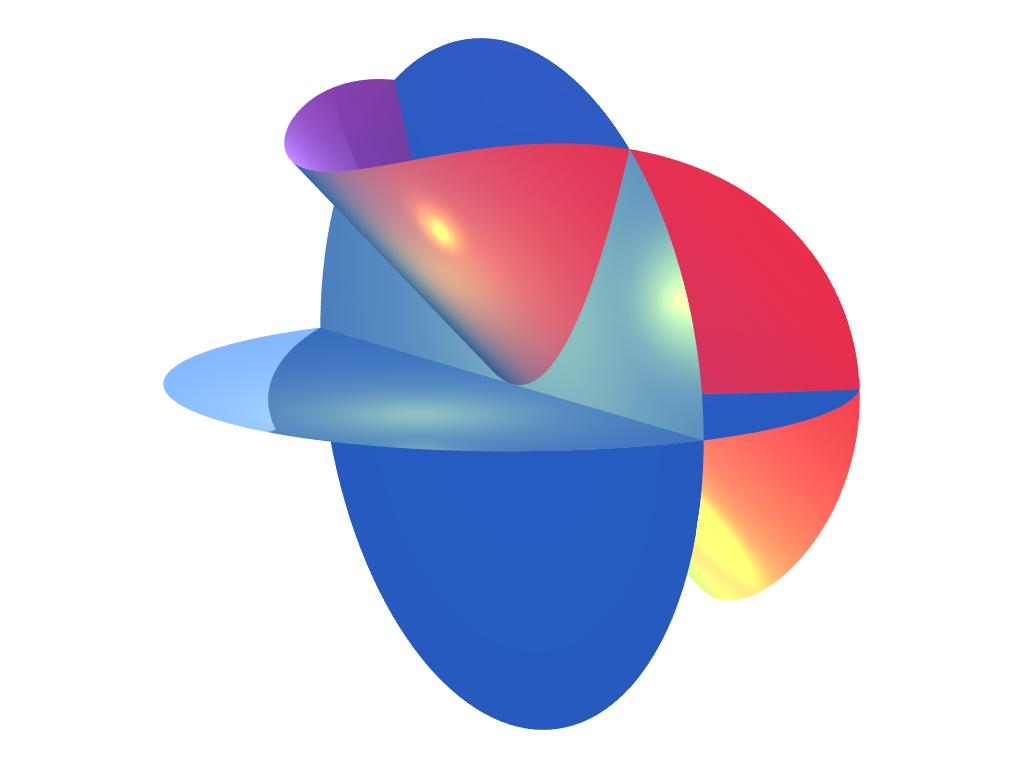}
\end{center}
\caption{ \label{Fig:tuelle} T\"ulle, defined by $h=xz(x+z-y^2)$, has normal crossings outside the origin but is not free at the origin.}
\end{figure}

\begin{lemma} \label{Lem:koordinatenwechselnc}
Denote by $(S,D)$ a complex manifold of dimension $n$ together with a divisor $D \subseteq S$, and let $(D,p)=\bigcup_{i=1}^m (D_i,p)$ be the decomposition of $D$ into irreducible components at a
point $p$ in $S$. Suppose that $h=h_1 \cdots h_m$ is the local equation of $D$ at $p$. Then  $D$ has normal crossings at $p$ if and only if the
$dh_i/h_i$ are part of a basis, whose elements are closed, of the form
$\omega_1=dh_1/h_1, \ldots, \omega_m=dh_m/h_m, \omega_{m+1}=df_{m+1}, \ldots, \omega_n=df_n$ of $\Omega^1_{S,p}(\log D)$, that is, 
$$\frac{dh_1}{h_1} \wedge \cdots \wedge \frac{dh_m}{h_m} \wedge df_{m+1} \wedge \cdots \wedge df_n=\frac{c}{h} \cdot dx_1 \wedge \cdots \wedge x_n,$$ 
 where the $f_{i}$ are some suitable elements in $\mc{O}_{S,p}$ and $c \in \mc{O}_{S,p}^*$.
\end{lemma}

\begin{proof} 
If $D$ has normal crossings at $p$ then one can find coordinates $x=(x_1, \ldots, x_n)$ such that $h=x_1 \cdots x_m$ is the defining equation of $D$ at $p$. Then clearly 
$$\frac{dx_1}{x_1}, \ldots, \frac{dx_m}{x_m}, dx_{m+1}, \ldots, dx_n$$ 
form a basis of $\Omega^1_{S,p}(\log D).$ 
Conversely, suppose that $\frac{dh_1}{h_1} \wedge \ldots \wedge \frac{dh_m}{h_m} \wedge df_{m+1} \wedge \cdots \wedge df_n=c/h \cdot dx_1 \wedge \ldots \wedge x_n$. This means that the Jacobian matrix
of
the $h_1, \ldots, h_m,$ $f_{m+1}, \ldots, f_n$ has determinant $c \in \mc{O}_{S,p}^*$. By the implicit function Theorem the $h_i$ and the $f_i$ are complex coordinates at $p$. Then, by definition $D$
has normal crossings at $p$.
\end{proof}

\begin{lemma}  \label{Lem:holomorphErz}
Let $D \subseteq S$ be a divisor in a complex manifold $S$ with $\dim S=n$. Suppose that $D$ is free at a point $p \in S$ and $\Omega^1_{S,p}(\log D)$ has a basis $\omega_1, \ldots, \omega_n$ such
that $\omega_1, \ldots, \omega_k$, $k <n$ are in $\Omega^1_{S,p}$. Then one can find a local isomorphism $(D,p) \cong (D',p') \times (\C^k,0)$, where $(D',p')$ is in $(\C^{n-k},p')$.
\end{lemma}

\begin{proof}
Since $\Omega^1_{S,p}(\log D)$ is free with basis $\omega_1, \ldots, \omega_n$, there is a unique basis $\delta_1, \ldots, \delta_n$ of $\Der_{S,p}(\log D)$ satisfying $\omega_i \cdot
\delta_j=\delta_{ij}$. Suppose that $1 \leq i \leq k$. For all coefficients of $\omega_i=\sum_{j=1}^n w_{ij}dx_j$ and $\delta_i=\sum_{j=1}^n d_{ij}
\D_{x_j}$ are holomorphic, one obtains the equation
$$ 1= \sum_{j=1}^n w_{ij} d_{ij}.$$
Since $\mc{O}_{S,p}$ is a local ring, at least one $w_{ij} d_{ij}$, w.l.o.g., for $j=1$, is invertible in $\mc{O}_{S,p}$, which implies  $d_{i1} \in \mc{O}^*_{S,p}$. Applying $\delta_i$ to $h$, the
local defining equation of $D$, yields $\D_{x_1}h \in (h, \D_{x_2}h, \ldots, \D_{x_n}h)$. The triviality Lemma \cite[3.5]{Saito80}  implies that $D$ is locally isomorphic to some $D' \times \C$. Applying this construction to the remaining
$\omega_i$, one arrives at $(D,p) \cong (D',p') \times (\C^k,0)$.
\end{proof}

\begin{lemma} \label{Lem:closed} 
Let $(D,p)=\bigcup_{i=1}^m (D_i,p)$ be  given by the reduced equation $h=h_1 \cdots
h_m$ and let $\omega \in \Omega^1_{S,p}(\log D)$ be a closed form. Then: \\
(i) The residue of $\omega$ along each branch $D_i$ is constant, that is,
$\rho(\omega)|_{D_i}=c_i $ with $c_i \in \C$ for $i=1, \ldots, m$. \\
(ii) $\omega$ can be represented as $\omega=\sum_{i=1}^m c_i dh_i /h_i+
\xi$, where $c_i \in \C$ and $\xi \in \Omega^1_{S,p}$ is closed. \\
(iii) If the residue of $\omega$ along at least one branch $D_i$ is non-zero,
then $\omega$ can be represented as
\[ \omega= \sum_{i=1}^m c_i \frac{dh_i'}{h_i'}, \quad c_i \in \C, \]
with $h_i'=u_ih_i$ and $u_i \in \mc{O}_{S,p}^*$. Note that $h_i'$ also defines
$D_i$ and that $h'=h_1' \cdots h_m'$ also defines
$D$ near $p$.
\end{lemma}

\begin{proof}
(i) is shown in the proof of \cite[Thm.~2.9]{Saito80}. \\
(ii):  Follows from the fact that $\rho(\omega)=0$ if and only if $\omega$ is holomorphic and the resulting exact sequence, see \cite[2.5]{Saito80}.\\
(iii): Suppose that $\omega=\sum_{i=1}^m c_i \frac{dh_i}{h_i} + \xi$, with $c_i \in
\C$, is a closed logarithmic form. Since we consider germs of differential forms, one can assume (Poincar\'e's Lemma) that $\xi=df$ for some $f \in \mc{O}_{S,p}$. 
Now assume that the residue $\rho(\omega)|_{D_1}$ is non-zero. Define $h_1':=h_1 \exp (f / c_1 )$. Then
$h_1'h_2 \cdots h_m$ also defines $D$ because multiplying with a unit does not
change the zero-set locally at $p$. The following holds:
\[ c_1 \frac{dh_1'}{h_1'} = c_1 \frac{dh_1}{h_1} + df=c_1
\frac{dh_1}{h_1} + \xi. \]
Hence we have $\omega= c_1 dh_1'/h_1' + \sum_{i=2}^m c_i dh_i/h_i$.
\end{proof}

\begin{lemma} \label{Lem:closedholomorphic} 
Let $(D,p)=\bigcup_{i=1}^m (D_i,p)$ be free at $p$ and let $\Omega^1_{S,p}(\log
D)$ have a basis $\omega_1, \ldots, \omega_n$ consisting of closed forms. Then
$m \leq n$ and maximally $n-m$ elements $\omega_i$ of this basis are holomorphic
forms. 
\end{lemma}

\begin{proof} 
From Lemma \ref{Lem:closed} it follows that each closed basis element
$\omega_i$ can be represented as $\omega_i=\sum_{j=1}^mc_{ij}
dh_j/h_j + df_i$ with $df_i \in \Omega^1_{S,p}$ and $c_{ij} \in \C$ for
$j=1, \ldots, m$. First suppose that $m > n$. By Saito's criterion one knows that $\bigwedge_{i=1}^n \omega_i= \frac{c}{h_1 \cdots h_m} \cdot dx_1 \wedge
\ldots \wedge dx_n$ with $c \in \mc{O}^*_{S,p}$. This means that the $n$-form
$\bigwedge_{i=1}^n \omega_i$ has a simple pole at $h_1 \cdots h_m$. But forming the
wedge product of the $\omega_i$ of the above form we obtain (by a simple computation)
 $\bigwedge_{i=1}^n \omega_i= \frac{g}{h_1 \cdots h_m} \cdot dx_1 \wedge \cdots
\wedge dx_n$ with $g \in (h_1, \ldots, h_m) \subseteq \mf{m}$. Thus $g$ is not
invertible, which is a  contradiction to Saito's criterion. \\
For the second assertion suppose that $\omega_i=df_i$, $f_i \in \mc{O}_{S,p}$
for $i=m, \ldots, n$ are holomorphic, that is, the basis contains $n-m+1$ closed holomorphic elements. An application of Lemma \ref{Lem:holomorphErz} yields an isomorphism $(D,p) \cong (D',0) \times
(\C^{n-m+1},0)$ with $(D',0) \subseteq (\C^{m-1},0)$. This means that $D'$ would be a free divisor with $m$ irreducible components and with a basis of closed forms in an $m-1$ dimensional manifold.
Contradiction to the first assertion of this lemma.
\end{proof}

\begin{proposition} \label{Prop:closednc}
Let $(D,p)=\bigcup_{i=1}^m (D_i,p)$ be free at $p$ and let
$\Omega^1_{S,p}(\log D)$ have a basis consisting of closed forms $\omega_1,
\ldots, \omega_n$. Then $m \leq n$ and $\omega_i$ can be chosen as $\omega_i=
dh_i'/h_i'$ where $h_i'=f_i h_i$ with $f_i \in \mc{O}_{S,p}^*$ for $i=1,
\ldots, m$ and $\omega_i=df_i$  with $f_i \in \mc{O}_{S,p}$ holomorphic for
$i=m+1, \ldots, n$. 
\end{proposition}

\begin{proof}
From Lemma \ref{Lem:closedholomorphic} it follows that $m \leq n$ and from Lemma
\ref{Lem:closed} that $(\omega_1, \ldots, \omega_n)$ can be
represented as 
\[ (\omega_1, \ldots, \omega_n)^T= \begin{pmatrix} C & 0 \\ 0 & I_{n-m} \end{pmatrix} \begin{pmatrix} \frac{dh}{\underline{h}} \\ 0 \end{pmatrix} + \begin{pmatrix} \underline{\xi} \\ \underline{df}
\end{pmatrix}\]  
with $C$ an $m \times m$-matrix with entries in $\C$,
$\frac{dh}{\underline{h}}=(\frac{dh_1}{h_1}, \ldots, \frac{dh_m}{h_m})^T$,
$\underline{\xi}=(\xi_1, \ldots, \xi_m)^T$ with $\xi_i \in \Omega^1_{S,p}$ and
$\underline{df}=(df_{m+1}, \ldots, df_n)^T$ with $f_i \in \mc{O}_{S,p}$.
Elementary linear algebra computations and an application of Lemma \ref{Lem:closedholomorphic} yield that $C \in GL_m (\C)$. Thus one can
assume that $(\omega_1, \ldots, \omega_m)$ is of the form $(\frac{dh_1}{h_1}+
\xi'_1, \ldots, \frac{dh_m}{h_m} + \xi'_m)$, where $\underline{\xi'}=M\underline{\xi}$ for some invertible matrix $M \in GL_m(\C)$. As in Lemma \ref{Lem:closed} write $\omega_i= dh_i'/h_i'$, where for $\xi'_i=df_i/f_i,
f_i \in \mc{O}_{S,p}^*$ one has $h_{i}'=f_i h_i$ for $i=1, \ldots, m$. The change of one $h_i$ does
not affect the others. The functions $h_i'$
also define the divisor $D$ at $p$. 
The assertion of the proposition follows.
\end{proof}

\begin{proposition} \label{Thm:ncequivclosedforms}
Denote by $(S,D)$ a complex manifold with $\dim S=n \geq 2$ together with a divisor $D \subseteq S$ and let $p \in S$ be a point. The following conditions are equivalent: \\
(i) $D$ has  normal crossings at $p$. \\
(ii) $\Omega^1_{S,p}(\log D)$ is free and has a basis of closed forms. \\
(iii) $\Der_{S,p}(\log D)$ is free and there exists a basis $\delta^1, \ldots, \delta^n$ of $\Der_{S,p}(\log D)$ such that $[ \delta^i, \delta^j]=0$ for all $i, j \in \{ 1, \ldots, n\}$.
\end{proposition}

\begin{proof}
(i) $\Rightarrow$ (ii) is a simple computation (cf. proof of Lemma  \ref{Lem:koordinatenwechselnc}). \\
Conversely, suppose that $\Omega^1_{S,p}(\log D)$ has a basis of closed forms.
By Proposition \ref{Prop:closednc} we can assume that $(D,p)=\bigcup_{i=1}^m (D_i,p)$ has $m \leq n$ irreducible components and that the closed basis of $\Omega_{S,p}(\log D)$ is
of the form $dh_{1}/h_{1}, \ldots, dh_{m}/h_{m}, df_{m+1},$ $
\ldots, df_{n}$, where $h_i$ is the reduced function corresponding to  the component $(D_i,p)$. By Lemma \ref{Lem:koordinatenwechselnc} the existence of a  closed basis of this form is
equivalent to $(D,p)$ having normal crossings. \\
It remains to show $(ii) \Leftrightarrow (iii)$: We have 
\begin{equation} \label{Equ:cartandiff}
 d\omega (\xi^1, \xi^2)=\xi^1 (\omega (\xi^2))- \xi^2 (\omega(\xi^1))- \omega([\xi^1, \xi^2]),
\end{equation}
where $\omega$ is a differential 1-form and $\xi^1, \xi^2$ are vector fields (for this well-known formula see e.g. \cite{Lang99}).
First, suppose that $[\delta^i, \delta^j]=0$ for all pairs $(i,j)$.  Plugging any basis elements $\delta^i, \delta^j$ into a basis element $\omega_k$ yields $d\omega_k (\delta^i, \delta^j)=\delta^i (\delta_{jk})-
\delta^j (\delta_{ik})- \omega (0) =0$. Hence any basis element $\omega_k$ is closed. 
Conversely, if each $\omega_k$ is closed, it follows from (\ref{Equ:cartandiff}) that $\omega_k ([\delta^i, \delta^j])=0$. Since $\Der_{S,p}(\log D)$ is closed under $[\cdot, \cdot ]$ and the
$\delta$'s form a basis of $\Der_{S,p}(\log D)$, the equation $[\delta^i, \delta^j]=\sum_{k=1} g_k \delta^k$ holds for some $g_k \in \mc{O}_{S,p}$. Using the $\mc{O}_{S,p}$-linearity of $\omega_k$ we
obtain $g_k=0$ for any $k=1, \ldots, n$.
Since this equality holds for any $i,j,k$ it follows that $[\delta^i, \delta^j]=0$ for all pairs $(i,j)$.
\end{proof}

\begin{Qu}
\begin{enumerate}
\item{ Construct special bases: we ask for a constructive algorithm for a closed basis of
$\Omega_{S,p}^1(\log D)$ (resp. a basis of commuting fields of $\Der_{S,p}(\log D)$), which in the first place determines if there exists such a basis.}
\item{Construct  a minimal system of generators of $\Omega^q_{S,p}(\log D)$, in particular in the case where $(D,p)$ is
not free.}
\end{enumerate}
\end{Qu}

\section{Normal crossings and (dual) logarithmic residue}

In this section we present a characterization of  normal crossing divisor $D$ by its
logarithmic residue, denoted by $\rho(\Omega^1_{S,p}(\log D))$. It follows from results about the dual logarithmic residue, which was introduced by Granger and Schulze in \cite{GrangerSchulze11}. They showed that
with the dual logarithmic residue a question by K.~Saito \cite{Saito80,SaitoLe84} can be answered, also see Thm.~\ref{Thm:Saitoanswer}. \\
In \cite[2.8]{Saito80} it is shown that the logarithmic residue of $\Omega^1_{S,p}(\log D)$ always contains the ring of weakly holomorphic functions on $D$. So it is quite
natural to ask when the two rings are the same. For free divisors the answer is surprisingly simple (under the additional condition that the normalization of $D$ is smooth):
$\rho(\Omega^1_{S,p}(\log
D))=\pi_*\mc{O}_{\widetilde D,p}$ if and only if $(D,p)$ has normal crossings.  In general the equality is
equivalent to saying that $(D,p)$ has normal crossings in codimension 1 (see Thm.~\ref{Thm:Saitoanswer}).
\\
This section is organized as follows: first we consider examples of divisors $(D,p)$ with weakly holomorphic logarithmic residue. 
Then some properties of divisors with weakly holomorphic residues are studied. Finally we introduce the dual logarithmic residue in order to prove the theorem. \\

 Suppose that $D$ is a free divisor whose logarithmic residue $\rho(\Omega^1_S(\log D))$ is equal to $\pi_*\mc{O}_{\widetilde D}$. Recall that $\pi_*\mc{O}_{\widetilde D}$ is equal to the
normalization $\widetilde{\mc{O}}_D$ and also to the ring of weakly holomorphic functions on $D$.  Since we consider free divisors, it is possible to compute
$\rho(\Omega^1_S(\log D))$ and $\pi_*\mc{O}_{\widetilde D}$ explicitly with a computer algebra system: from a basis of $\Omega^1_{S,p}(\log D)$ the logarithmic residue can be
computed, and it is also possible to compute the normalization of $D$. However, computing normalizations is of high complexity, so we are confined to low dimensional examples.

\begin{example}
 Let $D \subseteq S$ with $\dim S=n$ be smooth at a point $p$. Then locally at $p$ we can find coordinates
$(x_1, \ldots, x_n)$ such that $D=\{x_1=0\}$. Since $\Omega^1_{S,p}(\log D)$ is
generated by $\frac{dx_1}{x_1}, dx_2, \ldots, dx_n$, the residue of a
logarithmic form $\omega=a_1 \frac{dx_1}{x_1}+ \sum_{i=2}^n a_i dx_i$ is just $a_1|_D$ and hence $\rho(\Omega^1_{S,p}(\log D)=\mc{O}_{D,p}$, also cf.~Thm.~\ref{Thm:closednc}.
\end{example}

\begin{example}
Consider the cusp $D$ in $\C^2$, given by $h=x^3-y^2$ with coordinate ring $\mc{O}_{D,0}=\C\{x,y\} / (x^3-y^2)$.  It is well known that $\widetilde{\mc{O}}_D=\C\{t\}$ with
$t=\frac{y}{x}$. A basis of $\Omega^1_{\C^2,0}(\log D)$ is $\omega_1=\frac{dh}{h}$ and $\omega_2=\frac{1}{h}(3ydx+2xdy)$. 
Here $\rho(\omega_2)=\frac{x}{y}=t^{-1}$ is clearly not in $\C\{t\}$. Thus it follows that  $\Omega^1(\log D) \supsetneq
\pi_*\mc{O}_{\widetilde D}$. 
\end{example}

\begin{example} 
(The \emph{4-lines}) In this example, the divisor $D$ is free but does not have normal crossings
outside an $(n-3)$-dimensional subset of $D$. Let $D$ be the divisor in $\C^3$ given at $p=(x,y,z)$ by $h=(x+y)y(x+2y)(x+y+yz)$.
The divisor $D$ is free, thus one can
compute a basis of $\Omega^1_{\C^3,p}(\log D)$, namely
\small{
 \begin{align*} 
\omega_1& =\frac{dh}{h} \\
\omega_2 &=   \text{ \small $ \dfrac{1}{4h} (y(zx+9yz+7x+7y) dx -x(zx+9yz+7x+7y)dy  -(x+y))y(2y+x)
dz)$} \\
\omega_3 &=\frac{1}{4h} (y(x+y+yz)dx  -x(x+y+yz)dy) 
\end{align*} }\normalsize
This basis is the dual to the basis of $\Der_{\C^3,p}(\log D)$ given in example 6.2 of \cite{CN02} (in different coordinates). Here
$\pi_*\mc{O}_{\widetilde{D},p}\cong \widetilde{\mc{O}}_{D,p}$ is
isomorphic to 
\small{
$$ \C\{x,y,z\}/(x+y) \oplus \C\{x,y,z\}/(y) \oplus
\C\{x,y,z\}/(x+2y) \oplus \C\{x,y,z\}/(x+y+yz).$$}\normalsize
Since $\dim(\{h=\D_yh=0\})=1$, we have $\rho(\omega_i)=\frac{a_{i2}}{\D_yh}$, where
$\omega_i=\frac{1}{h}(a_{i1}dx+a_{i2}dy+a_{i3}dz)$ for $i=1,2,3$. For example the computation of $\rho(\omega_3)|_{D_1}=-\frac{1}{4x}$ 
shows that the residue of $\omega_3$ is not
holomorphic in $\pi_*\mc{O}_{\widetilde{D}_1,p}$. Hence the inclusion
$\pi_*\mc{O}_{\widetilde D,p} \supsetneq \rho(\Omega^1_{S,p}(\log D))$ is strict.
\end{example}

\begin{example}
Consider the Whitney Umbrella $D$ in $\C^3$ given by $h=x^2-y^2z$. The normalization $\widetilde D$ is smooth at the origin and has coordinate ring $\pi_*\mc{O}_{\widetilde
D,0}=\C\{x,y,z,t\}/(x^2-y^2z,yt-x,z-t^2)\cong \C\{y,t\}$.  One can show that $\Omega^1_{\C^3,0}(\log D)$ is generated by $dh/h, \omega=(yzdx-xzdy-1/2xydz)/h$ and $dx,dy,dz$. Since
$\rho(\omega)=yz/2x=t/2$ it follows that $\rho(\Omega^1_{\C^3,0}(\log D))$ is holomorphic on the normalization. Note that $D$ is not free.
\end{example}

These examples lead to the following 

 \begin{proposition} \label{Thm:irreduzibel}
Let $(S,D)$ be a manifold of complex dimension $n$ together with a divisor $D \subseteq S$. Suppose that $D$ is free at $p$, that
$$\rho(\Omega^1_{S_p}(\log D))=\pi_*\mc{O}_{\widetilde D,p}$$ 
and that $(\widetilde D, \pi^{-1}(p))$ is smooth. Then $D$ has normal crossings at $p$. 
\end{proposition}

First we consider some general properties of divisors with weakly holomorphic residue, in particular we show that if $D$ is a free divisor in a complex manifold $S$ of dimension $n$, having $n$
irreducible components $D_i$ at a point $p$ and satisfying $\rho(\Omega^1_{S,p}(\log D))=\pi_*\mc{O}_{\widetilde D,p}$, then $D$ has normal crossings at $p$ (Corollary to Lemma \ref{Lem:dhinbasis}).
Then we introduce the dual logarithmic residue and prove Prop.~\ref{Thm:irreduzibel} (following Granger and Schulze).    \\

\subsection{Divisors with weakly holomorphic logarithmic residue}

Here we first show an analogue of Theorem \ref{Thm:closednc} $(i) \Leftrightarrow (iii)$. Then some properties of $\pi_*\mc{O}_{\widetilde D,p}$ are considered (Cohen--Macaulayness). Finally we show
how to choose ``good'' generators for $\Omega^1_{S,p}(\log D)$ if $\rho(\Omega^1_{S,p}(\log D))=\pi_*\mc{O}_{\widetilde D,p}$ and that $D$ is Euler--homogeneous in this case (Lemma
\ref{Lem:dhinbasis}).

\begin{proposition} \label{Prop:resweaklygenerators}
Let $(S,D)$ be a divisor $D$ in a complex manifold $S$ of dimension $n$.  Then the following are equivalent: \\
(i) $\Omega^1_{S,p}(\log D)={}_{\mc{O}_{S,p}}\!\langle \omega_1, \ldots \omega_k \rangle + \Omega^1_{S,p}$, such that $\rho(\omega_1), \ldots, \rho(\omega_k) \in \pi_*\mc{O}_{\widetilde D,p}$ generate
$\pi_*\mc{O}_{\widetilde D,p}$ as $\mc{O}_{D,p}$-module. \\
(ii) $\rho(\Omega^1_{S,p}(\log D))=\pi_*\mc{O}_{\widetilde D,p}$.
\end{proposition}

\begin{proof}
The implication $(i) \Rightarrow (ii)$ is clear, since $\rho$ is a sheaf homomorphism and $\rho(\Omega^1_{S,p})=0$. Suppose now that $\rho(\Omega^1_{S,p}(\log D))=\pi_*\mc{O}_{\widetilde D,p}$. The
normalization is a finitely generated $\mc{O}_{D,p}$-module, i.e., $\pi_*\mc{O}_{\widetilde D,p}=\sum_{i=1}^k \mc{O}_{D,p}\alpha_i$ for some $\alpha_i \in \pi_*\mc{O}_{\widetilde D,p}$. The sequence
\begin{equation} \xymatrix@1{ 0\ar[r] & \Omega_{S,p}^1\ar[r]  &  \Omega_{S,p}^1(\log D) \ar[r]^-{\rho} & \pi_*\mc{O}_{\widetilde D,p} \ar[r] & 0} \label{Diag:resexakthol}
\end{equation} 
is exact (cf.~\cite{Saito80}). Thus there exist some $\omega_i \in \Omega^1_{S,p}(\log D)$ such that $\rho(\omega_i)=\alpha_i$ for each $i=1, \ldots, k$. Now take any $\omega \in
\Omega^1_{S,p}(\log D)$. Then $\rho(\omega)=\sum_{i=1}^k a_i \rho(\omega_i)$ for some $a_i \in \mc{O}_{D,p}$. Choose some representatives of the $a_i \in \mc{O}_{S,p}$ and define $\omega' :=
\sum_{i=1}^ka_i \omega_i$. Clearly $\omega' \in \Omega^1_{S,p}(\log D)$ as well as $\omega - \omega'$. But $\rho(\omega - \omega')=0$, so  $\omega - \omega'$ is holomorphic. This shows that any $\omega \in \Omega^1_{S,p}(\log D)$ can be written as an $\mc{O}_{S,p}$-linear combination of the $\omega_i$ and some holomorphic form.
\end{proof}

\begin{lemma}   \label{Lem:basisweaklyholom}
Let $(S,D)$ be a divisor $D$ in a complex manifold $S$ of dimension $n$. Suppose that at a point $p$ the divisor is free and and $\rho(\Omega^1_{S,p}(\log D))=\pi_*\mc{O}_{\widetilde D,p}$. \\
(i) The ring $\pi_*\mc{O}_{\widetilde D,p}$ is Cohen--Macaulay. \\
(ii) If $D$  additionally is not smooth and does not contain a smooth factor at $p$, i.e., is not locally isomorphic to some Cartesian product $(D',p') \times (\C^k,0)$ for some $0 < k <n$, one may
assume that $\pi_*\mc{O}_{\widetilde D,p}$ is minimally generated by $n$ elements $\alpha_i$, where $\alpha_1=1$ and $\alpha_i \in \pi_*\mc{O}_{\widetilde D,p} \backslash \mc{O}_{D,p}$ for $i \geq 2$.
\\

\end{lemma}

\begin{proof}
(i): Under our assumptions, the exact sequence (\ref{Diag:resexakthol}) yields a free resolution of  $\pi_*\mc{O}_{\widetilde D,p}$ (as $\mc{O}_{S,p}$-module). Since we are working over a regular local ring, it follows that
$\mathrm{projdim}_{\mc{O}_{S,p}}(\pi_*\mc{O}_{\widetilde D,p}) \leq 1$. With the Auslander--Buchsbaum formula follows $\depth(\mf{m}_{S},\pi_*\mc{O}_{\widetilde D,p}) \geq n-1$ (where $\mf{m}_S$
denotes the maximal ideal of $\mc{O}_{S,p}$). 
Since the depth is stable under local homomorphisms, 
 $\depth(\mf{m}_D,\pi_*\mc{O}_{\widetilde D,p})\geq n-1$. First suppose that $(D,p)$ is
irreducible, then $\pi_*\mc{O}_{\widetilde D,p}$ is a local ring. Since then $\mc{O}_{D,p} \subseteq \pi_*\mc{O}_{\widetilde D,p}$ is a finite ring extension it follows 
 that the depth of $\pi_*\mc{O}_{\widetilde D,p}$ as an $\pi_*\mc{O}_{\widetilde D,p}$-module is also greater than or equal to $n-1$. Clearly, $\dim(\pi_*\mc{O}_{\widetilde D,p})=n-1$ and so the assertion
follows from the height-depth inequality. \\
If $(D,p)=\bigcup_{i=1}^m (D_i,p)$, where $(D_i,p)$ denote the irreducible components, then $\pi_*\mc{O}_{\widetilde D,p}=\bigoplus_{i=1}^m \pi_*\mc{O}_{\widetilde D_i,p}$ is a semi-local ring with
$m$ maximal ideals $\mf{m}_{\widetilde D_i}$, $i=1, \ldots, m$. Then $\pi_*\mc{O}_{\widetilde D,p}$ is Cohen--Macaulay if  $(\pi_*\mc{O}_{\widetilde
D,p})_{\mf{m}_{\widetilde D_i}}\cong \pi_*\mc{O}_{\widetilde D_i,p}$ is Cohen--Macaulay for all $i=1, \ldots, m$. But this follows from the irreducible case since
$\depth(\mf{m}_{S},\pi_*\mc{O}_{\widetilde D,p}) = \depth(\mf{m}_{S},\pi_*\mc{O}_{\widetilde D_i,p})$ for all $i=1, \ldots, m$. 
\\
(ii):  follows from Lemmata \ref{Prop:resweaklygenerators} and \ref{Lem:holomorphErz} and an application of the Lemma of Nakayama.  
\end{proof}

\begin{lemma}  \label{Lem:invertierbarnormalisierung}
Let  $D \subseteq S$ be a divisor in a complex manifold $S$.  If $b$ is an element in
$\mc{O}_{D,p}$ that is invertible in $\pi_*\mc{O}_{\widetilde D,p}$ then $b$ is
already invertible in $\mc{O}_{D,p}.$
\end{lemma}

\begin{proof}
Easy computation.
\end{proof}

\begin{lemma} \label{Lem:dhinbasis}  \label{Lem:dhinbasisallg}
Let $D \subseteq S$ be a divisor in a complex manifold $S$ of dimension $n$. Suppose that $\rho(\Omega^1_{S,p}(\log D))=\pi_* \mc{O}_{\widetilde D,p}$. Then $\frac{dh}{h} \in \Omega^1_{S,p}(\log D)$
can be chosen as an element of a minimal system of generators of $\Omega^1_{S,p}(\log D)$. If $(D,p)=\bigcup_{i=1}^m (D_i,p)$,  defined by $h=h_1 \cdots
h_m$ in $\mc{O}_{S,p}$ then the
$\frac{dh_i}{h_i}$ form part of a minimal system of generators of
$\Omega^1_{S,p}(\log D)$.
\end{lemma}

\begin{proof}
Since $\Omega^1_{S}(\log D)$ is a coherent analytic sheaf, the stalk $\Omega^1_{S,p}(\log D)$ has a finite minimal system of
generators $\omega_1, \ldots, \omega_k$ with $k \geq n$. One can write
$$ \frac{dh}{h}= \sum_{i=1}^k a_i \omega_i,$$
for some $a_i \in \mc{O}_{S,p}$. Taking residues yields
\begin{equation}1_{\pi_* \mc{O}_{\widetilde D,p}}= \sum_{i=1}^k a_i |_D \rho(\omega_i). \label{Equ:res} \end{equation}
First assume that $D$ is irreducible at $p$. Then $\pi_* \mc{O}_{\widetilde D,p}$ is a local ring  and at least one $a_i|_D$ has to be invertible in $\pi_*
\mc{O}_{\widetilde D,p}$. By Lemma
\ref{Lem:invertierbarnormalisierung} this $a_i|_D$ is already invertible in $\mc{O}_{D,p}$. Thus $a_i(0) \neq 0$ and hence $a_i$ is contained in $\mc{O}_{S,p}^*$. This implies that $\frac{dh}{h}$ can
be chosen as an element of a minimal system of generators of $\Omega^1_{S,p}(\log D)$ instead of $\omega_i$. \\
If $(D,p)=\bigcup_{i=1}^m (D_i,p)$ is the decomposition into irreducible components, equation (\ref{Equ:res}) reads as follows: 
$$ 1_{\pi_*\mc{O}_{\widetilde{D}},p}= \sum_{i=1}^k a_i|_{D} \rho (\omega_i)=
\sum_{j=1}^m\left(\sum_{i=1}^k a_i|_{D_j} \rho(\omega_i)|_{D_j}\right).$$
Since the sum of the $\pi_* \mc{O}_{\widetilde D_j,p}$ is direct, 
$$  1_{\pi_*\mc{O}_{\widetilde{D}_1},p}= \sum_{i=1}^k a_i|_{D_1}
\rho(\omega_i)|_{D_1}.$$
Like in the irreducible case, it follows that  $a_i|_{D_1}$, wlog. for  $i=1$, has to
be invertible in $\pi_*\mc{O}_{\widetilde{D}_1,p}$.
Also, it follows that a representative of $a_1|_{D_1}$ in
$\mc{O}_{S,p}$, namely $a_1$, is
invertible in $\mc{O}_{S,p}$, so we may
exchange $\omega_1$ and $\frac{dh_1}{h_1}$. For the remaining $\frac{dh_i}{h_i}$ a
similar argument is used. Thus all $\frac{dh_i}{h_i}$ can be chosen as part of
a minimal system of generators. Clearly, also $\frac{dh}{h}, \frac{dh_2}{h_2}, \ldots, \frac{dh_m}{h_m}$ can be selected as part of any minimal system of generators. 
\end{proof}

\begin{remark}   \label{Rmk:dhinbasisallg}
Consider $D$ with the assumptions of Lemma \ref{Lem:dhinbasis} and further suppose that $D$ is  free. Then the element $\frac{dh}{h}$ can be chosen as an
element of a basis of $\Omega^1_{S,p}(\log D)$. Then $D$ is Euler-homogeneous: one can find a $\delta \in \Der_{S,p}(\log D)$ such that $\delta \cdot \frac{dh}{h}=1$. Hence $\delta(h)=h$. So we have shown that all free divisors $D$
with $\pi_*\mc{O}_{\widetilde D,p}=\rho(\Omega^1_{S,p}(\log D))$ are
Euler-homogeneous at $p$.  
\end{remark}

\begin{corollary} \label{Cor:ncomponentsweakly}
Let $D$ be a divisor in a complex manifold $S$ of dimension $n$ and suppose that at a point $p$, $D$ has $n$ irreducible components $(D_i,p)$. If $D$ has weakly holomorphic residue and is free at $p$,
then $D$ has normal crossings at $p$. 
\end{corollary}

\begin{proof}
Use Lemma \ref{Lem:dhinbasis} and Saito's criterion.
\end{proof}

\subsection{The dual logarithmic residue}  \label{Sub:dualres}

The dual logarithmic residue was introduced by Granger and Schulze in \cite{GrangerSchulze11}.  It relates the Jacobian ideal of a divisor with the conductor ideal of the normalization.
Here it will be used for the proof of Prop.~\ref{Thm:irreduzibel}. \\

Let $(S,D)$ be a complex manifold $S$ of dimension $n$ together with a divisor $D$ that is locally at a point $p \in S$ given by $\{h=0\}$. Denote by $\pi: \widetilde D \ra D$ the normalization of
$D$. Here we will abbreviate $\mc{O}_{S,p}$ to $\mc{O}_S$ etc. By definition there is an exact sequence (cf. (\ref{Diag:resexakthol}))
\begin{equation} \xymatrix@1{ 0\ar[r] & \Omega_{S}^1\ar[r] &  \Omega_{S}^1(\log D) \ar[r]^-{\rho} & \rho(\Omega^1_S(\log D)) \ar[r] & 0}.\label{Diag:res_exact}
\end{equation}
By applying the functor $\mathrm{Hom}_{\mc{O}_S}(-,\mc{O}_S)$ to (\ref{Diag:res_exact}) one obtains
\begin{equation} \text{  \footnotesize  $\xymatrix@1{ 0\ar[r] & \Der_S(\log D) \ar[r] &  \Der_{S} \ar[r]^-{\sigma}  &  \rho(\Omega^1_S(\log D))^{\vee} \ar[r]  & \mathrm{Ext}^1_{\mc{O}_{S}}(\Omega^1_{S}(\log D),\mc{O}_{S}) \ar[r] & 0  .}$ }\label{Diag:dualresidue}
\end{equation} \normalsize
Here $-^{\vee}$ denotes $\mathrm{Hom}_{\mc{O}_D}(-, \mc{O}_D)$. By Lemma 4.5 of \cite{DSSWW11} one has  
$$ \text{ \small{$\mathrm{Ext}^1_{\mc{O}_S}(\rho(\Omega^1_S(\log D)),\mc{O}_S)=\mathrm{Hom}_{\mc{O}_D}(\rho(\Omega^1_S(\log D)),\mc{O}_D)=\rho(\Omega^1_S(\log D))^{\vee},$}}$$ \normalsize which explains the
third term on the right in (\ref{Diag:dualresidue}). Then  $\rho(\Omega^1_S(\log D))^{\vee}$ is called the \emph{dual logarithmic residue} and we denote it shortly by $\ms{R}_D^{\vee}$. \\
One can show (see \cite{GrangerSchulze11}) that $\rho(\Omega^1_{S}(\log D))={\widetilde J_h}^{\vee}$, where $\widetilde J_h$ is the ideal generated by $(\D_{x_1}h, \ldots, \D_{x_n}h) \subseteq
\mc{O}_D$, that is, the Jacobian ideal of $D$. 

\begin{proposition} \label{Prop:resconductor}
Let $D \subseteq S$ be free. If the logarithmic residue is weakly holomorphic, that is, $\rho(\Omega^1_{S,p}(\log D))=\pi_*\mc{O}_{\widetilde D,p}$, then $\widetilde J_h \subseteq \mc{O}_D$ is equal
to
the conductor ideal $C_D$. Conversely, if $\widetilde D$ is Cohen--Macaulay at $p$ and $\widetilde J_h=C_D$, then 
$$\rho(\Omega^1_{S,p}(\log D))=\pi_*\mc{O}_{\widetilde D,p}.$$
\end{proposition}

\begin{proof} See \cite{GrangerSchulze11}. 
\end{proof}

\begin{lemma} \label{Lem:normalsmooth}
Let $D \subseteq S$ be a divisor in a complex manifold of dimension $n$. Suppose that $D$ is free at $p$, $\rho(\Omega^1_{S,p}(\log D))=\pi_*\mc{O}_{\widetilde D,p}$ and that
$(D,p)=\bigcup_{i=1}^m(D_i,p)$, where each irreducible component $D_i$ is normal. Then all $(D_i,p)$ are smooth and $(D,p)$ has normal crossings.
\end{lemma}

\begin{proof}
Since all irreducible components are normal, it follows that $\rho(\Omega^1_{S,p}(\log D))=\bigoplus_{i=1}^m \mc{O}_{D_i,p}$. By Theorem \ref{Thm:closednc} and Proposition \ref{Thm:ncequivclosedforms} $(D,p)$ is
a normal crossing singularity.
\end{proof}

\begin{proof}[Proof of Prop.~\ref{Thm:irreduzibel}] 
Using Prop.~\ref{Prop:resconductor} it follows (similar to Thm.~\ref{Thm:radikalJacobi}) from Piene's Theorem that $\Omega^1_{\widetilde D / D}=0$. 
By \cite[VI, Prop.~1.18, Prop.~1.20]{AltmanKleiman} (localization to an irreducible component $D_i$ and base change) it follows that $\Omega^1_{\widetilde D_i / D_i}=0$ for all $i=1, \ldots, m$. Then
using Lemma \ref{Lem:normalsmooth}, the remaining proof is similar to the one of Thm.~\ref{Thm:radikalJacobi}.
\end{proof}

Theorem \ref{Thm:closednc} suggests that $\rho( \Omega^1_{S}(\log D))$, the residue of
logarithmic 1-forms,  is directly
related to the geometry of the divisor $D$. K.~Saito has considered the relationship between the logarithmic residue and the local fundamental group of the complement of the divisor. He asked the following, cf.~\cite[(2.12)]{Saito80}:

\begin{Qu}[K. Saito] \label{Conj:Saito}
 Let $(S,D)$ be a manifold with $\dim S=n$ together with a divisor $D \subseteq S$ and let $p$ be a point on $D$. Are the following equivalent?
\\
(i) The local fundamental group $\pi_{1,q}(S \backslash D)$ for $q$ near $p$ is
abelian. \\
(ii) There exists an $(n-3)$-dimensional analytic subset $Z$ of $D$, such that
$D \backslash Z$ has only normal crossing singularities in a neighbourhood of
$p$. \\
(iii)  $\rho(\Omega^1_{S,p}(\log D))=\pi_*\mc{O}_{\widetilde D,p}$. 
\end{Qu}

The implications $(i) \Rightarrow (ii) \Rightarrow (iii)$ were proven in
\cite{Saito80}.  In 1985  L\^e and Saito \cite{SaitoLe84} gave a topological proof of the
equivalence of $(i)$ and $(ii)$. The implication $(iii) \Rightarrow (ii)$ was only recently proven by Granger and Schulze \cite{GrangerSchulze11}. Below is a proof using our Proposition
\ref{Thm:irreduzibel}. Hence all three conditions are equivalent. There
seems to be no obvious link between the residue and the fundamental
group, and nobody seems to have studied how to prove directly that $(i)$ is equivalent
to $(iii)$.

\begin{theorem}[Granger--Schulze] \label{Thm:Saitoanswer}
Let $(S,D)$ be a complex manifold together with a divisor $D \subseteq S$. If the logarithmic residue $\rho(\Omega^1_S(\log D))=\pi_*\mc{O}_{\widetilde D}$, then $D$ has normal crossings in
codimension 1.
\end{theorem}

\begin{proof}
By a Theorem of Scheja \cite{Scheja64} (also see \cite{SiuTrautmann71}), applied to $\mc{O}_{\Sing D}$, the divisor $D$ is free outside an analytic subset $Z \subseteq S$ of codimension at least $2$
in $D$. Since $\rho(\Omega^1_{S,p}(\log D))=\pi_*\mc{O}_{\widetilde D,p}$ for all $p
\in S$  and  $\widetilde D$ is by definition smooth in codimension 1 it follows from Prop.~\ref{Thm:irreduzibel} that $D$ has normal crossings outside an analytic set of codimension 2 in
$D$.
\end{proof}

Similar to conjecture \ref{Conj:gorensteinnormal}  one asks 

\begin{Qu}
Is a free divisor that has normal crossings in codimension 1 already a normal crossing divisor?
\end{Qu}

\subsection*{Acknowledgements}
 I thank my advisor Herwig Hauser for introducing me to this topic and for many comments and suggestions. I also thank David Mond, Luis Narv\'aez and Mathias Schulze for
several discussions, suggestions and comments. Moreover, ideas for this work originated from discussions with Alexander G.~Aleksandrov, Michel Granger, Jan Schepers, Bernard Teissier and Orlando
Villamayor, whose help shall be acknowledged.

\bibliographystyle{alpha}      
\bibliography{biblioNC}   

\end{document}